\documentclass[10pt,a4paper]{article}

\usepackage[utf8]{inputenc}
\usepackage[T1]{fontenc}
\usepackage{lmodern}
\usepackage{amsmath,amssymb,amsfonts,amsthm,mathtools}
\usepackage{geometry}
\usepackage{enumitem}
\usepackage{microtype}
\usepackage{hyperref}
\usepackage{cite}
\usepackage{xcolor}
\hypersetup{hidelinks}

\newtheorem{theorem}{Theorem}[section]
\newtheorem{proposition}[theorem]{Proposition}
\newtheorem{lemma}[theorem]{Lemma}
\newtheorem{corollary}[theorem]{Corollary}
\theoremstyle{definition}
\newtheorem{definition}[theorem]{Definition}
\newtheorem{remark}[theorem]{Remark}

\newcommand{\C}{\mathbb C}
\newcommand{\N}{\mathbb N}
\newcommand{\Der}{\operatorname{Der}}
\newcommand{\Inn}{\operatorname{Inn}}

\newcommand{\TwoLocDer}{\operatorname{2LocDer}}
\newcommand{\ad}{\operatorname{ad}}
\newcommand{\g}{\mathfrak g}
\newcommand{\slc}{\mathfrak{sl}_2(\C)}
\newcommand{\sod}{\mathfrak{so}_d(\C)}

\title{\textbf{2-Local Derivations on Conformal Galilei Lie Algebras and Their Central Extensions}}
\author{
Bekzod Sultonboyev\thanks{Department of Applied Mathematics and Mathematical Physics, Urgench State University, Urgench, Uzbekistan. Email: \texttt{bekzodbekbobomurodovich@gmail.com}}
\and
Bakhtiyor Yusupov\thanks{V.~I.~Romanovskiy Institute of Mathematics, Academy of Sciences of Uzbekistan, Tashkent, Uzbekistan; Department of Applied Mathematics and Mathematical Physics, Urgench State University, Urgench, Uzbekistan. Email: \texttt{baxtiyor\_yusupov\_93@mail.ru}}
}
\date{}

\begin{document}
\maketitle

\begin{abstract}
We classify 2-local derivations on the finite-dimensional conformal Galilei Lie algebras in three natural settings: the algebra without central extension, the mass central extension, and the exotic central extension.  No additivity or linearity is assumed in the definition of a 2-local derivation.  We prove that in every case each 2-local derivation is an ordinary derivation.  For the unextended algebra, the proof combines characteristic-quotient normalization with separating vectors, invariant bilinear forms, rank-two mixed zero anchors, and a semisimple-perturbation argument for the standard module.  For the mass central extension with $2\ell\ge3$, the residual map is first reduced to the center and then eliminated by a symplectic-anchor construction on the Heisenberg radical; a dimension argument in the centralizer of a regular semisimple element simultaneously handles singular reductive actions and the additional central shear that occurs in spatial dimension two.  The endpoint $\ell=\tfrac12$ is the Schr\"odinger algebra; a self-contained symplectic-anchor proof includes all spatial dimensions, with explicit treatment of dimensions one and two.  For the exotic extension, after center normalization we use conjugate middle-weight positions and an affine middle anchor to eliminate every central translate, including the one-sided middle-weight configurations.  Consequently the set of 2-local derivations coincides with the derivation algebra for all admissible parameters, and hence every 2-local derivation in these families is automatically linear.
\end{abstract}


\section{Introduction}
Throughout this paper, $\N=\{1,2,3,\ldots\}$ denotes the set of positive integers.
The purpose of this paper is to solve the two-point local-to-global problem for conformal Galilei Lie algebras and their standard one-dimensional central extensions.  A map $\Delta$ on a Lie algebra is called a 2-local derivation if, for every pair $x,y$, there exists a derivation $D_{x,y}$ such that $\Delta(x)=D_{x,y}(x)$ and $\Delta(y)=D_{x,y}(y)$.  The map $\Delta$ is not assumed to be additive or linear.  This absence of linearity is precisely what makes the passage from two-point data to a single global derivation non-trivial.

The notion of a 2-local derivation was introduced by \v{S}emrl in the operator-algebra setting \cite{Semrl}.  In finite-dimensional Lie theory, Ayupov, Kudaybergenov and Rakhimov proved that every 2-local derivation on a finite-dimensional semisimple Lie algebra over an algebraically closed field of characteristic zero is a derivation \cite{AKR}.  Ayupov and Yusupov studied 2-local derivations of infinite-dimensional Lie algebras \cite{AYInfinite}.  Subsequent work established analogous rigidity phenomena for generalized Witt algebras \cite{AKYWitt}, locally simple Lie algebras \cite{AKYLocSimple}, and several classes of Leibniz algebras \cite{AKHY,AKYp}.  These results also show that two-point rigidity is genuinely structural: non-semisimple algebras need not behave uniformly, and proper local or 2-local mappings may occur in other classes. Recent work of Atajonov, Vaisova and Yusupov extends the local and 2-local framework to $\frac12$-derivations of model Lie algebras and their solvable extensions \cite{AVYModelHalf}.

Schr\"odinger-type algebras are especially close to the present setting. Alauadinov and Yusupov investigated local derivations on the Schr\"odinger Lie algebra in $(n+1)$-dimensional space-time \cite{AYSchrodinger2026}.  Wang and Tang treated the Schr\"odinger algebra in $(2+1)$-dimensional space-time \cite{WangTang}; Tang and Wang subsequently proved that every 2-local derivation on the Schr\"odinger algebra in arbitrary $(n+1)$-dimensional space-time is a derivation \cite{TangWangJLT}.  Yao and Yu studied 2-local derivations of the $n$-th Schr\"odinger algebra \cite{YaoYu}, while Zhao and Cheng considered the one-dimensional mass-central conformal Galilei algebra \cite{ZhaoChengCG}.  The infinite-dimensional planar Galilean conformal algebra was studied by Chen and He \cite{ChenHe}.  These results motivate a uniform treatment of the full finite-dimensional conformal Galilei family, including both types of central extension.  Relative to these results, the present work treats the unextended family for arbitrary spatial dimension and admissible $\ell$, extends the mass-central analysis to all spatial dimensions, and includes the finite-dimensional exotic extension in the same framework.  For the Schr\"odinger endpoint $\ell=\tfrac12$, we also give a self-contained proof, including the additional central shear in spatial dimension two.

For a positive half-integer $\ell$ and a spatial dimension $d$, the algebra without central extension has the form
\[
\mathfrak g_\ell(d)=\bigl(\mathfrak{sl}_2(\mathbb C)\oplus\mathfrak{so}_d(\mathbb C)\bigr)\ltimes V_{\ell,d},
\]
where $V_{\ell,d}$ is an abelian module.  When $\ell\in\mathbb N-\tfrac12$ there is a mass central extension in every spatial dimension, and when $d=2$ and $\ell\in\mathbb N$ there is an exotic central extension.  In the central cases the abelian radical is replaced by a Heisenberg-type ideal whose bracket couples complementary $\mathfrak{sl}_2$-levels.  The representation-theoretic and dynamical background of these algebras can be found, for example, in \cite{AizawaIsaac,AizawaIsaacKimura,Andrzejewski,GalajinskyMasterov}.  Local derivations on the one-dimensional mass-central conformal Galilei algebra were studied in \cite{AYCG}.

The direct predecessor of the present paper is the work of Atajonov, Sultonboyev and Yusupov \cite{ASYLocalCG}, which studies linear local derivations on these same three families and determines their derivation algebras. Its classification includes proper local derivations on $\mathfrak g_\ell(1)$ when $m=2\ell\ge3$ is odd; the additional space has dimension $m-1$. Our result proves that this exceptional local freedom disappears under the two-point condition.

A linear 2-local derivation is a local derivation, but an arbitrary 2-local map is not initially known to be linear. Thus the local classification in \cite{ASYLocalCG} does not by itself establish our theorem. Our task is to derive linearity from pairwise interpolation and eliminate the exceptional one-dimensional freedom. The derivation decompositions are recalled from \cite{ASYLocalCG}; the new contribution is the two-point rigidity argument.

The main result of the present paper is that all three families are rigid with respect to 2-local derivations.  More precisely, for every admissible parameter we prove
\[
\TwoLocDer\mathfrak g_\ell(d)=\Der\mathfrak g_\ell(d),
\]
\[
\TwoLocDer\widehat{\mathfrak g}^{\,M}_\ell(d)
 =\Der\widehat{\mathfrak g}^{\,M}_\ell(d),
\qquad \ell\in\mathbb N-\tfrac12,
\]
and
\[
\TwoLocDer\widehat{\mathfrak g}^{\,\Theta}_\ell(2)
 =\Der\widehat{\mathfrak g}^{\,\Theta}_\ell(2),
\qquad \ell\in\mathbb N.
\]
No linearity hypothesis is imposed on the 2-local map.

Several different mechanisms are needed to obtain a uniform proof.  In the non-central algebra, characteristic quotients first remove the reductive part of a 2-local map.  For higher $\mathfrak{sl}_2$-weights, separating vectors and mixed zero anchors synchronize the scalar and rotation parameters; the adjoint module is controlled by its invariant symmetric form.  The standard module is more delicate.  Rank-two anchors settle the zero-rotation slice, while a semisimple-rotation argument and a perturbation lemma eliminate arbitrary skew-symmetric rotation components.

The central extensions require a different final step.  The value at the center fixes the scalar outer coefficient and, after quotient rigidity, the residual map becomes center-valued.  For the mass extension we use the non-degenerate alternating form on the Heisenberg radical.  A regular semisimple zero anchor and a dimension argument in its centralizer allow us to choose an auxiliary radical vector that kills all inner central contributions; in spatial dimension two the same construction is adjusted to cancel the central shear.  For the exotic extension, the $p$- and $q$-strings are first placed in a convenient conjugate middle-weight position.  An affine anchor of the form $h'+\tau s+ap_\ell'+bq_\ell'$ then forces a single quadratic relation whose parameter can always be chosen so that the remaining central value vanishes.

The paper is arranged as follows.  Section~\ref{sec:algebras} recalls the three conformal Galilei Lie algebras.  Section~\ref{sec:der2local} records their derivation algebras and the general 2-local reductions.  Section~\ref{sec:noncentral1d} treats the one-dimensional non-central algebra.  Section~\ref{sec:noncentralhigher} proves higher-dimensional non-central rigidity, including the standard-module semisimple-perturbation argument.  Section~\ref{sec:centralreductions} treats the two central extensions and proves their complete rigidity.  Section~\ref{sec:mainclassification} collects the classification, dimension consequences, and automatic linearity statement.

\section{The three conformal Galilei Lie algebras}\label{sec:algebras}
In this section we fix notation for the three finite-dimensional Lie algebras that will be studied throughout the paper.  We use the same Levi factor in all cases and emphasize the change from an abelian radical to a Heisenberg-type radical after central extension.

\subsection{The algebra without central extension}
We begin with the unextended conformal Galilei algebra, whose abelian radical is the source of the main two-local compatibility problem.

Let $d\in\N$ and $\ell\in\frac12\N$.  The Lie algebra $\g_\ell(d)$ has basis
\[
\{D,H,C,M_{ij}=-M_{ji},P_i^{(n)}\mid 1\le i,j\le d,\ 0\le n\le2\ell\}
\]
and non-zero brackets
\begin{equation}\label{eq:cg-basic}
\begin{aligned}
[D,H]&=2H, & [D,C]&=-2C, & [C,H]&=D,\\
[M_{ij},M_{kl}]&=-\delta_{ik}M_{jl}-\delta_{jl}M_{ik}+\delta_{il}M_{jk}+\delta_{jk}M_{il},\\
[H,P_i^{(n)}]&=-nP_i^{(n-1)}, & [D,P_i^{(n)}]&=2(\ell-n)P_i^{(n)},\\
[C,P_i^{(n)}]&=(2\ell-n)P_i^{(n+1)},\\
[M_{ij},P_k^{(n)}]&=-\delta_{ik}P_j^{(n)}+\delta_{jk}P_i^{(n)}.
\end{aligned}
\end{equation}
We set $P_i^{(-1)}=P_i^{(2\ell+1)}=0$ and put
\[
V_{\ell,d}=\operatorname{span}_{\C}\{P_i^{(n)}\}.
\]
Then
\[
\g_\ell(d)=\bigl(\slc\oplus\sod\bigr)\ltimes V_{\ell,d},
\qquad [V_{\ell,d},V_{\ell,d}]=0.
\]

\subsection{The mass central extension}
We next pass to the mass central extension, where complementary weight levels are paired through one central generator.

Assume $\ell\in\N-\frac12$.  The mass extension $\widehat{\g}^{\,M}_\ell(d)$ is obtained by adjoining a central element $M$ and the brackets
\begin{equation}\label{eq:mass2local}
[P_i^{(m)},P_j^{(n)}]
=\delta_{ij}\delta_{m+n,2\ell}I_mM,
\qquad
I_m=(-1)^{m+\ell+\frac12}(2\ell-m)!m!.
\end{equation}
The ideal
\[
\mathfrak h^M_{\ell,d}=V_{\ell,d}\oplus\C M
\]
is two-step nilpotent and is the nilradical of the mass extension.

\subsection{The exotic central extension}
The exotic extension exists only in two spatial dimensions and uses an alternating spatial cocycle.

Assume $d=2$ and $\ell\in\N$.  The exotic extension $\widehat{\g}^{\,\Theta}_\ell(2)$ is obtained by adjoining a central element $\Theta$ and the brackets
\begin{equation}\label{eq:exotic2local}
[P_i^{(m)},P_j^{(n)}]
=\varepsilon_{ij}\delta_{m+n,2\ell}J_m\Theta,
\qquad
J_m=(-1)^m(2\ell-m)!m!,
\end{equation}
where $\varepsilon_{12}=1$.

For later calculations we use
\[
h=D,\qquad e=-H,\qquad f=C,\qquad s=-\mathrm iM_{12},
\]
\[
p_n=P_1^{(n)}+\mathrm iP_2^{(n)},\qquad
q_n=P_1^{(n)}-\mathrm iP_2^{(n)},\qquad z=-2\mathrm i\Theta.
\]
Then
\begin{align}
[h,e]&=2e,& [h,f]&=-2f,&[e,f]&=h,\label{eq:sl2}\\
[h,p_n]&=2(\ell-n)p_n,&[e,p_n]&=np_{n-1},&[f,p_n]&=(2\ell-n)p_{n+1},\\
[h,q_n]&=2(\ell-n)q_n,&[e,q_n]&=nq_{n-1},&[f,q_n]&=(2\ell-n)q_{n+1},\\
[s,p_n]&=p_n,&[s,q_n]&=-q_n,&[p_m,q_n]&=\delta_{m+n,2\ell}(-1)^m(2\ell-m)!m!\,z.
\end{align}

\section{Derivations and general 2-local reductions}\label{sec:der2local}
This section records the derivation decompositions needed in the sequel and then develops general tools that apply simultaneously to all three algebras.  The key points are homogeneity, stability under subtraction of a global derivation, and descent to a characteristic quotient.

\subsection{Derivation decompositions}
The derivation decompositions for these families are established in \cite{ASYLocalCG}. We state them in the notation used here, since all two-point representing derivations will be written in these forms. Their proofs are given in \cite{ASYLocalCG}.

For the non-central algebra define
\[
\delta_0|_{\slc\oplus\sod}=0,
\qquad
\delta_0(P_i^{(n)})=P_i^{(n)}.
\]
For the mass extension define
\[
\delta_M|_{\slc\oplus\sod}=0,
\qquad
\delta_M(P_i^{(n)})=P_i^{(n)},
\qquad
\delta_M(M)=2M.
\]
When $d=2$, with $s=-\mathrm iM_{12}$, define the central shear
\[
\eta_M(s)=M,
\qquad
\eta_M(\text{all other basis elements})=0.
\]
For the exotic extension define
\[
\delta_\Theta(e)=\delta_\Theta(f)=\delta_\Theta(h)=\delta_\Theta(s)=0,
\quad
\delta_\Theta(p_n)=p_n,
\quad
\delta_\Theta(q_n)=q_n,
\quad
\delta_\Theta(z)=2z,
\]
and
\[
\eta_\Theta(s)=z,
\qquad
\eta_\Theta(\text{all other basis elements})=0.
\]

\begin{proposition}[\cite{ASYLocalCG}]\label{prop:derivation-decompositions}
The derivation algebras are
\begin{equation}\label{eq:der-noncentral}
\Der\g_\ell(d)=\Inn\g_\ell(d)\oplus\C\delta_0,
\end{equation}
\begin{equation}\label{eq:der-mass}
\Der\widehat{\g}^{\,M}_\ell(d)=
\begin{cases}
\Inn\widehat{\g}^{\,M}_\ell(d)\oplus\C\delta_M,& d\ne2,\\
\Inn\widehat{\g}^{\,M}_\ell(2)\oplus\C\delta_M\oplus\C\eta_M,& d=2,
\end{cases}
\end{equation}
and
\begin{equation}\label{eq:der-exotic}
\Der\widehat{\g}^{\,\Theta}_\ell(2)
=\Inn\widehat{\g}^{\,\Theta}_\ell(2)
\oplus\C\delta_\Theta\oplus\C\eta_\Theta.
\end{equation}
\end{proposition}

The one-dimensional mass-central decomposition agrees with the form used in \cite{AYCG,ZhaoChengCG}.

\subsection{Basic properties of 2-local derivations}
We now record three elementary facts that will be used repeatedly to normalize a 2-local derivation.

\begin{definition}
Let $L$ be a Lie algebra. A mapping $\Delta:L\to L$, not assumed to be linear, is called a \emph{2-local derivation} if for every $x,y\in L$ there exists a derivation $D_{x,y}\in\Der L$ such that
\[
\Delta(x)=D_{x,y}(x),\qquad
\Delta(y)=D_{x,y}(y).
\]
\end{definition}

\begin{lemma}\label{lem:homogeneous}
Every 2-local derivation is homogeneous:
\[
\Delta(\lambda x)=\lambda\Delta(x)
\qquad (x\in L,\ \lambda\in\C).
\]
In particular, $\Delta(0)=0$.
\end{lemma}
\begin{proof}
Apply the definition to the pair $x,\lambda x$.  If $D_{x,\lambda x}$ is the corresponding derivation, then
\[
\Delta(\lambda x)=D_{x,\lambda x}(\lambda x)
=\lambda D_{x,\lambda x}(x)=\lambda\Delta(x).
\]
\end{proof}

\begin{lemma}\label{lem:subtract}
If $\Delta$ is a 2-local derivation and $D\in\Der L$, then $\Delta-D$ is again a 2-local derivation.
\end{lemma}
\begin{proof}
For a pair $x,y$, replace the representing derivation $D_{x,y}$ by $D_{x,y}-D$.
\end{proof}

\begin{lemma}\label{lem:quotient}
Let $I$ be a characteristic ideal of a Lie algebra $L$ and let $\Delta$ be a 2-local derivation on $L$. Then
\[
\Delta(I)\subseteq I,
\]
and $\Delta$ induces a well-defined 2-local derivation
\[
\overline\Delta:L/I\longrightarrow L/I.
\]
\end{lemma}
\begin{proof}
For $x\in I$, take a representing derivation for the pair $(x,0)$. Since every derivation preserves the characteristic ideal $I$, we obtain $\Delta(x)\in I$.

If $x-y\in I$, choose a derivation $D_{x,y}$ representing the pair $(x,y)$. Then
\[
\Delta(x)-\Delta(y)=D_{x,y}(x-y)\in I,
\]
so $\overline\Delta(x+I)=\Delta(x)+I$ is well-defined. Finally, every $D_{x,y}$ induces a derivation of $L/I$, and therefore $\overline\Delta$ satisfies the two-local condition on the quotient.
\end{proof}

\begin{lemma}\label{lem:d2-reductive-quotient}
Fix one conformal Galilei algebra $L$ considered here in spatial dimension two, and identify its quotient by the abelian radical or Heisenberg nilradical with $Q=\mathfrak{sl}_2\oplus\mathbb Cs$. Suppose a map $\Phi:Q\to Q$ has the following property: for every $x,y\in Q$, there is a derivation of this fixed algebra $L$ whose induced derivation $\overline D$ satisfies $\Phi(x)=\overline D(x)$ and $\Phi(y)=\overline D(y)$.  Then there exists $a\in\mathfrak{sl}_2$ such that
\[
\Phi(x+\alpha s)=[a,x]
\qquad(x\in\mathfrak{sl}_2,\ \alpha\in\mathbb C).
\]
In particular, after subtracting the inner derivation induced by $a$, the whole quotient map is zero, not only its restriction to the $\mathfrak{sl}_2$-factor.
\end{lemma}
\begin{proof}
By Proposition~\ref{prop:derivation-decompositions}, every derivation of a full conformal Galilei algebra induces on $Q$ a derivation that is inner on $\mathfrak{sl}_2$ and annihilates $s$.  Hence the restriction of $\Phi$ to $\mathfrak{sl}_2$ is a 2-local derivation on $\mathfrak{sl}_2$.  By \cite{AKR} there is $a\in\mathfrak{sl}_2$ such that
\[
\Phi(x)=[a,x]\qquad(x\in\mathfrak{sl}_2).
\]
Fix $x\in\mathfrak{sl}_2$ and $\alpha\in\mathbb C$.  If $x\ne0$, choose one induced derivation representing the pair $(x,x+\alpha s)$.  It has the form $\operatorname{ad}b$ on the $\mathfrak{sl}_2$-factor and kills $s$, so
\[
\Phi(x)=[b,x],\qquad
\Phi(x+\alpha s)=[b,x].
\]
Thus $\Phi(x+\alpha s)=\Phi(x)=[a,x]$.  If $x=0$, every induced derivation annihilates $\alpha s$, hence $\Phi(\alpha s)=0$.  This proves the formula.
\end{proof}

\begin{corollary}\label{cor:levi-normalization}
Let $L$ be one of the three conformal Galilei algebras above and let $R$ denote its abelian radical or Heisenberg nilradical.  After subtracting a suitable global derivation, a 2-local derivation may be assumed to induce the zero map on the full quotient $L/R$.
\end{corollary}
\begin{proof}
The radicals under consideration are characteristic.  Hence Lemma~\ref{lem:quotient} applies.  For $d=1$ and for $d\ge3$, the non-central quotient is semisimple, and the result follows from the rigidity theorem of Ayupov--Kudaybergenov--Rakhimov \cite{AKR}.  When $d=2$, the quotient is $\slc\oplus\C s$, and Lemma~\ref{lem:d2-reductive-quotient} shows that the entire induced quotient map is one inner derivation on the $\slc$-factor and is zero on the $s$-direction.  Subtracting its inner lift therefore kills the full quotient map.  The central extensions have the same reductive quotient.
\end{proof}

\section{The one-dimensional non-central algebra}\label{sec:noncentral1d}
We now turn to the first genuinely new two-local calculation.  In one spatial dimension the algebra is a semidirect product of $\mathfrak{sl}_2$ with a single irreducible module, which allows the two-point condition to be translated into explicit weight-space restrictions.

Put $m=2\ell$ and write
\[
\g_\ell(1)=\slc\ltimes U_m,
\qquad
U_m=\operatorname{span}\{u_0,u_1,\ldots,u_m\}.
\]
We use the standard action
\begin{equation}\label{eq:Um-action}
[h,u_k]=(m-2k)u_k,
\qquad
[e,u_k]=ku_{k-1},
\qquad
[f,u_k]=(m-k)u_{k+1},
\end{equation}
where $u_{-1}=u_{m+1}=0$.  The radical $U_m$ is abelian.

Every derivation has the form
\begin{equation}\label{eq:D-one-string}
D=\ad(a+u)+\lambda\delta_0,
\qquad a\in\slc,\ u\in U_m,\ \lambda\in\C.
\end{equation}
On the two summands this means
\begin{align}
D(x)&=[a,x]-\rho_m(x)u, &&x\in\slc,\label{eq:D-on-S}\\
D(v)&=\rho_m(a)v+\lambda v, &&v\in U_m.\label{eq:D-on-V}
\end{align}

\subsection{Normalization on the Levi factor}
The first step is to remove the quotient derivation and then use the pair $(e,f)$ to normalize simultaneously at the two nilpotent generators of $\mathfrak{sl}_2$.

\begin{proposition}\label{prop:normalize-ef}
Let $\Delta$ be a 2-local derivation on $\g_\ell(1)$.  There exists $D\in\Der\g_\ell(1)$ such that, after replacing $\Delta$ by $\Delta-D$,
\[
\Delta(\g_\ell(1))\subseteq U_m,
\qquad
\Delta(e)=\Delta(f)=0.
\]
\end{proposition}
\begin{proof}
By Lemma~\ref{lem:quotient}, $\Delta$ induces a 2-local derivation on
\[
\g_\ell(1)/U_m\cong\slc.
\]
By \cite{AKR} the induced map is a derivation, hence an inner derivation.  Subtracting a suitable inner derivation and using Lemma~\ref{lem:subtract}, we may assume that $\Delta(\g_\ell(1))\subseteq U_m$.

Choose a representing derivation for the pair $(e,f)$ and write it as in \eqref{eq:D-one-string}. Since $\Delta(e)$ and $\Delta(f)$ have no $\slc$-components, we have
\[
[a,e]=[a,f]=0.
\]
The intersection of the centralizers of $e$ and $f$ in $\slc$ is zero, so $a=0$.  Consequently
\[
\Delta(e)=-\rho_m(e)u,
\qquad
\Delta(f)=-\rho_m(f)u
\]
for one and the same $u\in U_m$.  Subtracting $\ad(u)$ gives the required normalization.
\end{proof}

\begin{lemma}\label{lem:sl2-zero}
Under the normalization of Proposition~\ref{prop:normalize-ef},
\[
\Delta(x)=0\qquad\text{for every }x\in\slc.
\]
\end{lemma}
\begin{proof}
We first prove $\Delta(h)=0$.  Apply the two-local condition to $(e,h)$.  A representing derivation $D=\ad(a+u)+\lambda\delta_0$ satisfies $D(e)=0$.  Separating the $\slc$- and $U_m$-components gives
\[
a\in\C e,
\qquad
u\in\ker\rho_m(e)=\C u_0.
\]
Since $\Delta(h)\in U_m$, the $\slc$-component of $D(h)$ must vanish, hence $a=0$. Therefore
\[
\Delta(h)\in\C u_0.
\]
Using the pair $(f,h)$ in the same way yields $\Delta(h)\in\C u_m$.  Since $u_0$ and $u_m$ are linearly independent, $\Delta(h)=0$.

Now write
\[
x=ae+bf+ch\in\slc.
\]
For the pair $(e,x)$, the condition $D(e)=0$ implies $a_D\in\C e$ and $u_D\in\C u_0$.  If $x\notin\C e$, then the $\slc$-component of $D(x)=\Delta(x)\in U_m$ forces $a_D=0$, and hence
\[
\Delta(x)\in\C\rho_m(x)u_0.
\]
Similarly, if $x\notin\C f$, then
\[
\Delta(x)\in\C\rho_m(x)u_m.
\]
For $m\ge3$,
\[
\rho_m(x)u_0=m(cu_0+bu_1),
\qquad
\rho_m(x)u_m=m(au_{m-1}-cu_m).
\]
The two displayed vectors have disjoint supports in the basis $u_0,\ldots,u_m$, so their one-dimensional spans have zero intersection. Thus $\Delta(x)=0$ whenever $x\notin\C e\cup\C f$, and homogeneity gives the same conclusion on $\C e\cup\C f$.

For $m=2$, the only additional possibility occurs when $c=0$ and $ab\ne0$, in which case both previous lines are contained in $\C u_1$.  Using the pair $(h,x)$, the equality $D(h)=0$ implies that the $\slc$-part of the representing derivation lies in $\C h$ and that its radical part lies in $\ker\rho_2(h)=\C u_1$.  Since $x\notin\C h$, the $\slc$-part is zero, so
\[
\Delta(x)\in\C\rho_2(x)u_1
=\C(au_0+bu_2).
\]
Its intersection with $\C u_1$ is zero.

Finally, if $m=1$, then $\rho_1(h)$ is invertible.  For any $x\notin\C h$, the pair $(h,x)$ and the equality $D(h)=0$ force both the radical part and the $\mathfrak{sl}_2$-part of the representing derivation to vanish; the outer scalar derivation is zero on $\mathfrak{sl}_2$. Hence $\Delta(x)=0$.  Homogeneity handles $x\in\C h$.  This proves the assertion for all $m\ge1$.
\end{proof}

\subsection{The action on the abelian radical}
We next show that, once the Levi factor has been normalized to zero, the two-point condition strongly restricts the radical action.  For modules of dimension at least four, one generic radical vector separates the full four-dimensional operator space generated by $\mathfrak{sl}_2$ and the identity.

\begin{lemma}\label{lem:separating-vector}
Assume $m\ge3$ and put
\[
w=u_0+u_1+u_m.
\]
If
\[
A=\alpha\rho_m(e)+\beta\rho_m(f)+\gamma\rho_m(h)+\lambda I_{U_m}
\]
satisfies $Aw=0$, then $A=0$.
\end{lemma}
\begin{proof}
For $m\ge4$, the coefficient of $u_{m-1}$ in $Aw$ is $m\alpha$, while the coefficient of $u_2$ is $(m-1)\beta$. Hence $\alpha=\beta=0$. The coefficients of $u_0$ and $u_m$ then give
\[
m\gamma+\lambda=0,
\qquad
-m\gamma+\lambda=0,
\]
so $\gamma=\lambda=0$.

For $m=3$, the coefficients of $u_0,u_1,u_2,u_3$ give
\[
\alpha+3\gamma+\lambda=0,
\qquad
3\beta+\gamma+\lambda=0,
\]
\[
3\alpha+2\beta=0,
\qquad
-3\gamma+\lambda=0.
\]
Substitution yields $\alpha=\beta=\gamma=\lambda=0$.
\end{proof}

\begin{proposition}\label{prop:radical-scalar}
Assume $m\ge3$ and let $\Delta$ be a 2-local derivation satisfying
\[
\Delta|_{\slc}=0.
\]
Then there exists $c\in\C$ such that
\[
\Delta(v)=cv\qquad\text{for every }v\in U_m.
\]
\end{proposition}
\begin{proof}
Let $w=u_0+u_1+u_m$.  Using the pair $(e,w)$ and the equality $\Delta(e)=0$, formula \eqref{eq:D-on-V} gives
\[
\Delta(w)\in\operatorname{span}\{w,\rho_m(e)w\}.
\]
Likewise, the pair $(f,w)$ gives
\[
\Delta(w)\in\operatorname{span}\{w,\rho_m(f)w\}.
\]
Now
\[
\rho_m(e)w=u_0+mu_{m-1},
\qquad
\rho_m(f)w=mu_1+(m-1)u_2.
\]
For $m\ge3$, the three vectors $w,\rho_m(e)w,\rho_m(f)w$ are linearly independent. Hence the intersection of the two displayed two-dimensional spaces is $\C w$, and therefore
\[
\Delta(w)=cw
\]
for some $c\in\C$.

Fix $v\in U_m$ and choose a representing derivation for the pair $(w,v)$.  Its action on $U_m$ has the form
\[
A=\rho_m(a)+\lambda I_{U_m},
\qquad a\in\slc.
\]
Since $Aw=cw$, the operator $A-cI$ annihilates $w$. Lemma~\ref{lem:separating-vector} gives $A=cI$. Thus
\[
\Delta(v)=Av=cv.
\]
Since $v$ was arbitrary, the assertion follows.
\end{proof}

\begin{theorem}\label{thm:one-string-reduction}
Let $m=2\ell\ge3$ and let $\Delta$ be a 2-local derivation on $\g_\ell(1)$.  Then there exists a derivation $D\in\Der\g_\ell(1)$ such that
\[
(\Delta-D)(x)=0
\qquad\text{for every }x\in\slc\cup U_m.
\]
\end{theorem}
\begin{proof}
Apply Proposition~\ref{prop:normalize-ef} and Lemma~\ref{lem:sl2-zero}.  Proposition~\ref{prop:radical-scalar} then shows that the normalized map equals $c\,\delta_0$ on $U_m$.  Subtracting this global derivation gives the conclusion.
\end{proof}

\subsection{Mixed elements and full rigidity}
We now control mixed elements.  We first record a basic intersection condition and then construct zero mixed anchors whose automorphic conjugates eliminate the remaining nonlinear freedom.

\begin{lemma}\label{lem:zero-basic-mixed-anchors}
Assume $m\ge3$ and let $\Delta$ be a 2-local derivation on $\g_\ell(1)$ such that
\[
\Delta|_{\slc}=0,
\qquad
\Delta|_{U_m}=0.
\]
With $w=u_0+u_1+u_m$, one has
\[
\Delta(e+w)=0,
\qquad
\Delta(f+w)=0.
\]
\end{lemma}
\begin{proof}
Put $z_e=e+w$.  First use the pair $(w,z_e)$.  Since $\Delta(w)=0$, a representing derivation
\[
D=\ad(a+u)+\lambda\delta_0
\]
satisfies
\[
(\rho_m(a)+\lambda I)w=0.
\]
Lemma~\ref{lem:separating-vector} gives $a=0$ and $\lambda=0$.  Hence
\[
\Delta(z_e)=-\rho_m(e)u\in\operatorname{Im}\rho_m(e).
\]

Next use the pair $(h,z_e)$.  Since $D(h)=0$, the two homogeneous components give
\[
a\in\C h,
\qquad
u\in\ker\rho_m(h).
\]
The value $\Delta(z_e)$ is already known to lie in $U_m$, so the $\slc$-component of $D(z_e)$ must vanish.  Therefore $[a,e]=0$, hence $a=0$.  We obtain
\[
\Delta(z_e)\in -\rho_m(e)\ker\rho_m(h)+\C w.
\]
If $m$ is odd, then $\ker\rho_m(h)=0$, so
\[
\Delta(z_e)\in\C w\cap\operatorname{Im}\rho_m(e)=0,
\]
because $w$ has a non-zero $u_m$-component while
$\operatorname{Im}\rho_m(e)=\operatorname{span}\{u_0,\ldots,u_{m-1}\}$.

Assume now that $m$ is even.  Then
\[
\ker\rho_m(h)=\C u_{m/2},
\]
and consequently
\[
\Delta(z_e)\in
\C u_{m/2-1}+\C w.
\]
Intersecting with $\operatorname{Im}\rho_m(e)$ removes the $w$-component, hence
\[
\Delta(z_e)\in\C u_{m/2-1}.
\]
Using instead the pair $(f,z_e)$, the equality $D(f)=0$ gives
\[
a\in\C f,
\qquad
u\in\ker\rho_m(f)=\C u_m.
\]
Again the $\slc$-component of $D(z_e)$ is zero, so $a=0$, and therefore
\[
\Delta(z_e)\in\C u_{m-1}+\C w.
\]
After intersecting with $\operatorname{Im}\rho_m(e)$ we get
\[
\Delta(z_e)\in\C u_{m-1}.
\]
Since $m\ge4$ in the even case,
$u_{m/2-1}$ and $u_{m-1}$ are distinct basis vectors.  Thus
$\Delta(z_e)=0$.

The proof of $\Delta(f+w)=0$ is symmetric.  One first obtains
\[
\Delta(f+w)\in\operatorname{Im}\rho_m(f),
\]
then uses the pairs $(h,f+w)$ and $(e,f+w)$.  If $m$ is odd, the first of these already gives zero.  If $m$ is even, they respectively yield the one-dimensional restrictions
\[
\C u_{m/2+1}
\qquad\text{and}\qquad
\C u_1,
\]
whose intersection is zero for $m\ge4$.
\end{proof}

\begin{lemma}\label{lem:conjugate-anchor}
Let $\Phi$ be an automorphism of $\g_\ell(1)$ preserving both $\slc$ and $U_m$.  If a 2-local derivation $\Delta$ satisfies
\[
\Delta|_{\slc}=0,
\qquad
\Delta|_{U_m}=0,
\]
then the conjugate map
\[
\Delta^{\Phi}=\Phi^{-1}\Delta\Phi
\]
is again a 2-local derivation with the same normalization.  In particular, under the assumptions of Lemma~\ref{lem:zero-basic-mixed-anchors},
\[
\Delta\bigl(\Phi(e+w)\bigr)=0,
\qquad
\Delta\bigl(\Phi(f+w)\bigr)=0.
\]
\end{lemma}
\begin{proof}
For any $x,y$ choose $D_{\Phi x,\Phi y}\in\Der\g_\ell(1)$ representing $\Delta$ at $\Phi x$ and $\Phi y$.  Then
\[
\Phi^{-1}D_{\Phi x,\Phi y}\Phi
\]
is a derivation representing $\Delta^\Phi$ at $x$ and $y$.  Thus $\Delta^\Phi$ is 2-local.  Since $\Phi$ preserves the two summands and $\Delta$ vanishes on each of them, so does $\Delta^\Phi$.  Lemma~\ref{lem:zero-basic-mixed-anchors} applied to $\Delta^\Phi$ gives the final assertion.
\end{proof}

\begin{theorem}\label{thm:one-string-full}
Let $m=2\ell\ge3$.  Then every 2-local derivation on
\[
\g_\ell(1)=\slc\ltimes U_m
\]
is a derivation.  Equivalently,
\[
\TwoLocDer\g_\ell(1)=\Der\g_\ell(1).
\]
\end{theorem}
\begin{proof}
By Theorem~\ref{thm:one-string-reduction}, after subtracting a global derivation we may assume
\begin{equation}\label{eq:normalized-zero-SV}
\Delta|_{\slc}=0,
\qquad
\Delta|_{U_m}=0.
\end{equation}
We prove that this normalized map is identically zero.

Let
\[
y=x+v,
\qquad x\in\slc,
\quad v\in U_m.
\]
Using the pair $(x,y)$ and $\Delta(x)=0$, we get
\[
\Delta(y)=D_{x,y}(v)\in U_m.
\]
Put again $w=u_0+u_1+u_m$ and
\[
z_e=e+w,
\qquad z_f=f+w.
\]
By Lemma~\ref{lem:zero-basic-mixed-anchors},
\[
\Delta(z_e)=\Delta(z_f)=0.
\]

Assume first that $x\notin\C e$.  For a representing derivation
\[
D=\ad(a+u)+\lambda\delta_0
\]
of the pair $(z_e,y)$, the equality $D(z_e)=0$ implies
\[
[a,e]=0,
\qquad
-\rho_m(e)u+\rho_m(a)w+\lambda w=0.
\]
Hence $a=\alpha e$.  Since the $u_m$-coefficient in the second equation comes only from $\lambda w$, we have $\lambda=0$, and then
\[
u=\alpha w+c u_0
\]
for some $c\in\C$.  The $\slc$-component of $D(y)=\Delta(y)$ equals
\[
[a,x]=\alpha[e,x].
\]
Because $\Delta(y)\in U_m$ and $x\notin\C e$, it follows that $\alpha=0$.  Therefore
\begin{equation}\label{eq:mixed-line-u0}
\Delta(y)\in\C\rho_m(x)u_0.
\end{equation}
Similarly, if $x\notin\C f$, the pair $(z_f,y)$ gives
\begin{equation}\label{eq:mixed-line-um}
\Delta(y)\in\C\rho_m(x)u_m.
\end{equation}
Write
\[
x=Ae+Bf+Ch.
\]
Then
\[
\rho_m(x)u_0=m(Cu_0+Bu_1),
\qquad
\rho_m(x)u_m=m(Au_{m-1}-Cu_m).
\]
For $m\ge3$ these vectors have disjoint supports.  Consequently, if
$x\notin\C e\cup\C f$, equations \eqref{eq:mixed-line-u0} and
\eqref{eq:mixed-line-um} imply $\Delta(y)=0$.

It remains to treat $x\in\C e\cup\C f$.  By homogeneity it is enough to consider $x=e$ and $x=f$.  Let first
\[
y=e+v.
\]
The pair $(z_f,y)$ gives
\begin{equation}\label{eq:e-mixed-first-line}
\Delta(y)\in\C\rho_m(e)u_m=\C u_{m-1}.
\end{equation}
Fix $t\ne0$ and let
\[
\Phi_t=\exp(t\,\ad f).
\]
This automorphism preserves $\slc$ and $U_m$.  Set
\[
s_t=\Phi_t(e),
\qquad
r_t=\Phi_t(w),
\qquad
z_t=s_t+r_t.
\]
By Lemma~\ref{lem:conjugate-anchor}, $\Delta(z_t)=0$.  Apply the two-local condition to $(z_t,y)$.  Since $\Delta(z_t)=0$ and $\Delta(y)\in U_m$, the $\slc$-components show that the $\slc$-part $a$ of the representing derivation belongs simultaneously to $\C s_t$ and to the centralizer of $e$.  As $s_t\notin\C e$ for $t\ne0$, we get $a=0$.  The equality $D(z_t)=0$ now becomes
\[
-\rho_m(s_t)u+\lambda r_t=0.
\]
Let
\[
G_t=\exp\bigl(t\rho_m(f)\bigr)
\]
be the restriction of $\Phi_t$ to $U_m$.  Then
\[
\rho_m(s_t)=G_t\rho_m(e)G_t^{-1},
\qquad
r_t=G_t(w).
\]
Since $w\notin\operatorname{Im}\rho_m(e)$, it follows that
$r_t\notin\operatorname{Im}\rho_m(s_t)$.  Hence $\lambda=0$ and
\[
u\in\ker\rho_m(s_t)=\C G_t(u_0).
\]
Thus
\begin{equation}\label{eq:e-mixed-second-line}
\Delta(y)\in\C\rho_m(e)G_t(u_0).
\end{equation}
Moreover,
\[
G_t(u_0)=\sum_{k=0}^{m}\binom{m}{k}t^k u_k.
\]
Consequently the coefficient of $u_0$ in
$\rho_m(e)G_t(u_0)$ equals $mt\ne0$.  Since $m\ge3$, this vector is not proportional to $u_{m-1}$.  Combining \eqref{eq:e-mixed-first-line} and \eqref{eq:e-mixed-second-line} gives
\[
\Delta(e+v)=0.
\]

The case $y=f+v$ is symmetric.  The pair $(z_e,y)$ first gives
\[
\Delta(y)\in\C\rho_m(f)u_0=\C u_1.
\]
Then, with
\[
\Psi_t=\exp(t\,\ad e),
\]
Lemma~\ref{lem:conjugate-anchor} gives a zero anchor
$\Psi_t(f+w)$.  Repeating the preceding argument yields
\[
\Delta(y)\in\C\rho_m(f)\Psi_t(u_m).
\]
The vector $\rho_m(f)\Psi_t(u_m)$ has a non-zero $u_m$-coefficient for $t\ne0$, so for $m\ge3$ it is not proportional to $u_1$.  Hence
\[
\Delta(f+v)=0.
\]

Finally, if $x=0$, then $y=v\in U_m$ and $\Delta(y)=0$ by
\eqref{eq:normalized-zero-SV}.  Thus $\Delta$ vanishes on every element of
$\g_\ell(1)$.  Undoing the initial subtraction proves that the original 2-local derivation is a derivation.
\end{proof}

\subsection{The adjoint-module case \texorpdfstring{$m=2$}{m=2}}
The separating-vector argument used above starts in module dimension four.  The remaining case $m=2$ has a different advantage: $U_2$ is the adjoint $\mathfrak{sl}_2$-module and therefore carries a nondegenerate invariant symmetric bilinear form.  This allows us to synchronize the scalar part of the radical action before returning to mixed elements.

Define a symmetric bilinear form $B$ on $U_2$ by
\begin{equation}\label{eq:B-U2}
B(u_0,u_2)=B(u_2,u_0)=1,
\qquad
B(u_1,u_1)=-\frac12,
\end{equation}
with all other pairings of basis vectors equal to zero.  A direct check from \eqref{eq:Um-action} gives
\begin{equation}\label{eq:B-invariance-U2}
B(\rho_2(a)v,w)+B(v,\rho_2(a)w)=0
\qquad
(a\in\slc,\ v,w\in U_2).
\end{equation}
The form $B$ is nondegenerate.

\begin{proposition}\label{prop:m2-radical-scalar}
Let $m=2$ and let $\Delta$ be a 2-local derivation on $\g_1(1)$ satisfying
\[
\Delta|_{\slc}=0.
\]
Then there exists $c\in\C$ such that
\[
\Delta(v)=cv
\qquad
\text{for every }v\in U_2.
\]
\end{proposition}
\begin{proof}
Since $U_2$ is isomorphic to the adjoint $\slc$-module, fix an $\slc$-module isomorphism
\[
\theta:\slc\longrightarrow U_2.
\]
For $0\ne v\in U_2$, write $v=\theta(x_v)$.  Then
\[
\rho_2(x_v)v=\theta([x_v,x_v])=0.
\]
Apply the two-local condition to the pair $(x_v,v)$.  A representing derivation has the form
\[
D=\ad(a+u)+\lambda\delta_0.
\]
Since $\Delta(x_v)=0$, the equality $D(x_v)=0$ gives
\[
[a,x_v]=0.
\]
The centralizer of a non-zero element of $\slc$ is one-dimensional, hence $a\in\C x_v$.  Therefore
\[
\Delta(v)=D(v)=(\rho_2(a)+\lambda I)v=\lambda v.
\]
Thus there is a homogeneous scalar function $\varphi$ on $U_2\setminus\{0\}$ such that
\[
\Delta(v)=\varphi(v)v.
\]

We show that $\varphi$ is constant.  Let $v,w$ be non-isotropic vectors for $B$.  For a representing derivation of the pair $(v,w)$, put
\[
A=\rho_2(a)+\lambda I.
\]
Then
\[
Av=\varphi(v)v,
\qquad
Aw=\varphi(w)w.
\]
By \eqref{eq:B-invariance-U2},
\[
B(Av,v)=\lambda B(v,v),
\qquad
B(Aw,w)=\lambda B(w,w).
\]
Hence
\[
\varphi(v)=\lambda=\varphi(w).
\]
Thus $\varphi$ has one common value, say $c$, on all non-isotropic vectors.

Now let $0\ne v$ be isotropic.  Choose a non-isotropic vector $w$ with
\[
B(v,w)\ne0.
\]
For the pair $(v,w)$, the previous paragraph forces the scalar coefficient of the representing operator $A=\rho_2(a)+\lambda I$ to satisfy $\lambda=c$.  Hence
\[
\rho_2(a)v=(\varphi(v)-c)v,
\qquad
\rho_2(a)w=0.
\]
Using \eqref{eq:B-invariance-U2},
\[
(\varphi(v)-c)B(v,w)
=B(\rho_2(a)v,w)
=-B(v,\rho_2(a)w)=0.
\]
Since $B(v,w)\ne0$, we obtain $\varphi(v)=c$.  Therefore
\[
\Delta(v)=cv
\]
for every $v\ne0$, and the same holds at $v=0$ by homogeneity.
\end{proof}

\begin{lemma}\label{lem:m2-zero-anchors}
Assume $m=2$ and let $\Delta$ be a 2-local derivation satisfying
\[
\Delta|_{\slc}=0,
\qquad
\Delta|_{U_2}=0.
\]
Put
\[
w=u_0+u_1+u_2.
\]
Then
\[
\Delta(e+w)=\Delta(f+w)=\Delta(h+w)=0.
\]
\end{lemma}
\begin{proof}
Set $z_e=e+w$.  The pair $(e,z_e)$ gives
\[
\Delta(z_e)\in\operatorname{span}\{w,\rho_2(e)w\},
\]
while the pair $(f,z_e)$ gives, after the vanishing of the $\slc$-component,
\[
\Delta(z_e)\in\operatorname{span}\{w,u_1\}.
\]
Since
\[
\rho_2(e)w=u_0+2u_1,
\]
the intersection of these two planes is $\C w$.  Thus
\[
\Delta(z_e)\in\C w.
\]
Now use the pair $(u_0,z_e)$.  If
\[
D=\ad(a+u)+\lambda\delta_0
\]
represents this pair, then $D(u_0)=0$.  Writing
\[
a=\alpha e+\beta f+\gamma h,
\]
we obtain
\[
(\rho_2(a)+\lambda I)u_0
=2\beta u_1+(2\gamma+\lambda)u_0=0.
\]
Hence $\beta=0$ and $\lambda=-2\gamma$.  Since $\Delta(z_e)$ has no $\slc$-component, the $\slc$-part of $D(z_e)$ gives
\[
[a,e]=2\gamma e=0.
\]
Thus $\gamma=\lambda=0$ and $a\in\C e$.  Consequently
\[
\Delta(z_e)\in\operatorname{Im}\rho_2(e)
=\operatorname{span}\{u_0,u_1\}.
\]
Its intersection with $\C w$ is zero, so $\Delta(z_e)=0$.

The proof of $\Delta(f+w)=0$ is symmetric, using the pair $(u_2,f+w)$ in the final step.

Finally put $z_h=h+w$.  The pair $(h,z_h)$ shows that $\Delta(z_h)\in U_2$.  The pairs $(e,z_h)$ and $(f,z_h)$ then imply
\[
\Delta(z_h)\in\C w.
\]
Use the pair $(u_1,z_h)$.  Since
\[
\rho_2(e)u_1=u_0,
\qquad
\rho_2(f)u_1=u_2,
\qquad
\rho_2(h)u_1=0,
\]
the equality $D(u_1)=0$ forces the $e$- and $f$-coefficients of the $\slc$-part of $D$ and also the scalar coefficient $\lambda$ to vanish.  Hence
\[
\Delta(z_h)\in\operatorname{Im}\rho_2(h)
=\operatorname{span}\{u_0,u_2\}.
\]
Again the intersection with $\C w$ is zero.  Therefore $\Delta(z_h)=0$.
\end{proof}

\begin{theorem}\label{thm:m2-full}
Every 2-local derivation on
\[
\g_1(1)=\slc\ltimes U_2
\]
is a derivation.
\end{theorem}
\begin{proof}
By Proposition~\ref{prop:normalize-ef} and Lemma~\ref{lem:sl2-zero}, after subtracting a global derivation we may assume
\[
\Delta|_{\slc}=0.
\]
Proposition~\ref{prop:m2-radical-scalar} then gives
\[
\Delta(v)=cv
\qquad(v\in U_2).
\]
Subtracting $c\delta_0$, we reduce to
\begin{equation}\label{eq:m2-normalized}
\Delta|_{\slc}=0,
\qquad
\Delta|_{U_2}=0.
\end{equation}

Let $y=x+v$ with
\[
x=Ae+Bf+Ch\in\slc,
\qquad v\in U_2.
\]
As in the proof of Theorem~\ref{thm:one-string-full}, the pair $(x,y)$ shows that
\[
\Delta(y)\in U_2.
\]
Put
\[
z_e=e+w,
\qquad
z_f=f+w,
\qquad
z_h=h+w,
\]
where $w=u_0+u_1+u_2$.  By Lemma~\ref{lem:m2-zero-anchors}, these are zero anchors.

If $x\notin\C e$, the pair $(z_e,y)$ gives
\[
\Delta(y)\in\C\rho_2(x)u_0
=\C(Cu_0+Bu_1).
\]
If $x\notin\C f$, the pair $(z_f,y)$ gives
\[
\Delta(y)\in\C\rho_2(x)u_2
=\C(Au_1-Cu_2).
\]
If $C\ne0$, these two lines have zero intersection, and therefore $\Delta(y)=0$.

Suppose $C=0$ and $x\notin\C e\cup\C f$.  Then $A B\ne0$ and the preceding two restrictions reduce to
\[
\Delta(y)\in\C u_1.
\]
Using the zero anchor $z_h$, the same coefficient comparison gives
\[
\Delta(y)\in\C\rho_2(x)u_1
=\C(Au_0+Bu_2).
\]
The two lines are disjoint, hence $\Delta(y)=0$.

It remains to treat $x\in\C e\cup\C f$.  Homogeneity reduces this to $x=e$ and $x=f$.  For $y=e+v$, the zero anchor $z_f$ gives
\[
\Delta(y)\in\C u_1.
\]
For $t\ne0$, put
\[
\Phi_t=\exp(t\,\ad f),
\qquad
G_t=\exp(t\rho_2(f)).
\]
By Lemma~\ref{lem:conjugate-anchor} applied with the $m=2$ zero-anchor statement, $\Phi_t(e+w)$ is also a zero anchor.  Repeating the endpoint argument in Theorem~\ref{thm:one-string-full} yields
\[
\Delta(y)\in\C\rho_2(e)G_t(u_0).
\]
Here
\[
G_t(u_0)=u_0+2t u_1+t^2u_2,
\]
so
\[
\rho_2(e)G_t(u_0)=2t u_0+2t^2u_1.
\]
For $t\ne0$ this vector is not proportional to $u_1$.  Thus $\Delta(e+v)=0$.

The case $y=f+v$ is symmetric: the first zero anchor gives $\Delta(y)\in\C u_1$, while conjugating $f+w$ by $\exp(t\,\ad e)$ produces a second line containing a non-zero $u_2$-component. Hence $\Delta(f+v)=0$.

Therefore the normalized map vanishes identically, and undoing the subtractions proves the theorem.
\end{proof}

\subsection{The standard-module case \texorpdfstring{$m=1$}{m=1}}
We finally treat the lowest weight.  Here $U_1$ is the two-dimensional standard $\mathfrak{sl}_2$-module.  The restriction of the derivation algebra to $U_1$ is the whole algebra $\operatorname{End}(U_1)$, so the separating-vector argument is unavailable.  The two-point condition with nilpotent elements of $\mathfrak{sl}_2$ nevertheless forces a single scalar action on the radical.

\begin{proposition}\label{prop:m1-radical-scalar}
Let $m=1$ and let $\Delta$ be a 2-local derivation on $\g_{1/2}(1)$ satisfying
\[
\Delta|_{\slc}=0.
\]
Then there exists $c\in\C$ such that
\[
\Delta(v)=cv
\qquad
(v\in U_1).
\]
\end{proposition}
\begin{proof}
Let $0\ne v\in U_1$.  There exists a non-zero nilpotent element $x_v\in\slc$ such that
\[
\ker\rho_1(x_v)=\C v.
\]
For the pair $(x_v,v)$, a representing derivation
\[
D=\ad(a+u)+\lambda\delta_0
\]
satisfies $D(x_v)=0$.  Hence
\[
[a,x_v]=0,
\]
so $a\in\C x_v$.  Since $\rho_1(x_v)v=0$,
\[
\Delta(v)=D(v)=\lambda v.
\]
Thus
\[
\Delta(v)=\varphi(v)v
\]
for a homogeneous scalar function $\varphi$ on $U_1\setminus\{0\}$.

Let $v,w$ be linearly independent.  Choose a nilpotent $x\in\slc$ and scale it so that
\[
\rho_1(x)v=0,
\qquad
\rho_1(x)w=v.
\]
Put
\[
y=x+v+w.
\]
The pair $(x,y)$ shows first that $\Delta(y)\in U_1$.  Now use the pair $(v,y)$.  Since $D(v)=\varphi(v)v$ and the $\slc$-component of $D(y)$ is zero, the $\slc$-part of the representing derivation commutes with $x$, hence equals a scalar multiple of $x$.  Because $xv=0$, its scalar outer coefficient is $\varphi(v)$.  The remaining term $-\rho_1(x)u$ lies in $\C v$, so the coefficient of $w$ in $\Delta(y)$ is exactly $\varphi(v)$.

Using instead the pair $(w,y)$, write again the $\slc$-part as a scalar multiple of $x$.  The equality
\[
D(w)=\varphi(w)w
\]
and the relation $\rho_1(x)w=v$ force that scalar multiple to be zero and the outer coefficient to equal $\varphi(w)$.  Therefore the coefficient of $w$ in the same vector $\Delta(y)$ equals $\varphi(w)$.  Hence
\[
\varphi(v)=\varphi(w).
\]
If two non-zero vectors are dependent, the same equality follows from homogeneity.  Thus $\varphi$ is constant on $U_1\setminus\{0\}$, proving the assertion.
\end{proof}

\begin{lemma}\label{lem:m1-zero-anchors}
Assume $m=1$ and let $\Delta$ be a 2-local derivation satisfying
\[
\Delta|_{\slc}=0,
\qquad
\Delta|_{U_1}=0.
\]
Put
\[
w=u_0+u_1.
\]
Then
\[
\Delta(e+w)=\Delta(f+w)=0.
\]
\end{lemma}
\begin{proof}
Set $z_e=e+w$.  The pair $(h,z_e)$ is especially effective because $\rho_1(h)$ is invertible.  From $D(h)=0$ we get
\[
a\in\C h,
\qquad
u=0.
\]
Since the value at $z_e$ has no $\slc$-component, $[a,e]=0$, so $a=0$.  Hence
\[
\Delta(z_e)\in\C w.
\]

Now use the pair $(u_0,z_e)$.  Write
\[
a=\alpha e+\beta f+\gamma h.
\]
The equality $D(u_0)=0$ gives
\[
\beta=0,
\qquad
\lambda=-\gamma,
\]
because for $m=1$,
\[
\rho_1(f)u_0=u_1,
\qquad
\rho_1(h)u_0=u_0.
\]
The $\slc$-component of $D(z_e)$ is $[a,e]=2\gamma e$, so $\gamma=\lambda=0$ and $a\in\C e$.  Consequently
\[
\Delta(z_e)\in\operatorname{Im}\rho_1(e)=\C u_0.
\]
Since $w=u_0+u_1$, the intersection $\C w\cap\C u_0$ is zero.  Therefore $\Delta(e+w)=0$.

The proof of $\Delta(f+w)=0$ is symmetric, using the pair $(u_1,f+w)$ in the second step.
\end{proof}

\begin{theorem}\label{thm:m1-full}
Every 2-local derivation on
\[
\g_{1/2}(1)=\slc\ltimes U_1
\]
is a derivation.
\end{theorem}
\begin{proof}
By Proposition~\ref{prop:normalize-ef} and Lemma~\ref{lem:sl2-zero}, after subtracting a global derivation we may assume that $\Delta$ vanishes on $\slc$.  Proposition~\ref{prop:m1-radical-scalar} and a further subtraction of $c\delta_0$ give
\begin{equation}\label{eq:m1-normalized}
\Delta|_{\slc}=0,
\qquad
\Delta|_{U_1}=0.
\end{equation}
We show that the normalized map is zero on mixed elements.

Let
\[
y=x+v,
\qquad x\in\slc,
\quad v\in U_1.
\]
The pair $(x,y)$ gives $\Delta(y)\in U_1$.  Put
\[
z_e=e+u_0+u_1,
\qquad
z_f=f+u_0+u_1.
\]
By Lemma~\ref{lem:m1-zero-anchors}, both are zero anchors.  Exactly as in the proof of Theorem~\ref{thm:one-string-full}, if $x\notin\C e$ then
\[
\Delta(y)\in\C\rho_1(x)u_0,
\]
and if $x\notin\C f$ then
\[
\Delta(y)\in\C\rho_1(x)u_1.
\]
If $\rho_1(x)$ is invertible, these two lines are distinct because they are the images of the independent vectors $u_0,u_1$ under an invertible operator.  Hence
\[
\Delta(x+v)=0
\qquad
\text{whenever }\det\rho_1(x)\ne0.
\]

It remains to treat non-zero nilpotent $x$.  We first consider $x=e$.  Write
\[
v=a u_0+b u_1.
\]
The zero anchor $z_f$ gives
\[
\Delta(e+v)\in\C u_0.
\]
If $b\ne0$, use the pair $(h,e+v)$.  Since $D(h)=0$ and $\rho_1(h)$ is invertible, the radical part of the representing derivation is zero.  The $\slc$-component at $e+v$ then forces the $\slc$-part to vanish as well.  Therefore
\[
\Delta(e+v)=\lambda v.
\]
Because $v$ has a non-zero $u_1$-component while $\Delta(e+v)\in\C u_0$, we get $\lambda=0$.

Suppose now that $b=0$, so $v=a u_0$.  The case $a=0$ is already contained in \eqref{eq:m1-normalized}; assume $a\ne0$.  Put
\[
r=-a u_1,
\qquad
z=h+r.
\]
Since $\rho_1(h)$ is invertible, the preceding invertible-$x$ argument gives $\Delta(z)=0$.  Use the pair $(z,e+a u_0)$.  The vanishing of the $\slc$-components forces the $\slc$-part of the representing derivation to be zero.  From $D(z)=0$ we obtain
\[
-\rho_1(h)u+\lambda r=0.
\]
As $\rho_1(h)u_1=-u_1$, this gives
\[
u=\lambda a u_1.
\]
Consequently
\[
D(e+a u_0)
=-\rho_1(e)(\lambda a u_1)+\lambda a u_0=0.
\]
Hence $\Delta(e+a u_0)=0$.  We have proved
\[
\Delta(e+v)=0
\qquad(v\in U_1).
\]

Finally, every non-zero nilpotent element of $\slc$ is a non-zero scalar multiple of an inner automorphic image of $e$.  Let $x=c\Phi(e)$ with $c\ne0$ and $\Phi$ induced by an inner automorphism of $\slc$.  The conjugate map
\[
\Delta^\Phi=\Phi^{-1}\Delta\Phi
\]
is again a normalized 2-local derivation.  By the result just proved,
\[
\Delta^\Phi(e+v')=0
\qquad(v'\in U_1).
\]
Using homogeneity and taking $v'=c^{-1}\Phi^{-1}(v)$ gives
\[
\Delta(x+v)=0.
\]
Together with the invertible and zero cases, this proves that the normalized map vanishes identically.  Therefore the original 2-local derivation is a derivation.
\end{proof}

\begin{corollary}\label{cor:one-dimensional-all}
For every positive half-integer $\ell$,
\[
\TwoLocDer\g_\ell(1)=\Der\g_\ell(1).
\]
\end{corollary}
\begin{proof}
The case $2\ell=1$ is Theorem~\ref{thm:m1-full}, the case $2\ell=2$ is Theorem~\ref{thm:m2-full}, and all $2\ell\ge3$ are covered by Theorem~\ref{thm:one-string-full}.
\end{proof}

\section{The higher-dimensional non-central algebra}\label{sec:noncentralhigher}
The one-dimensional argument shows that the two-point condition is strong enough to eliminate the extra local maps coming from a single irreducible $\mathfrak{sl}_2$-string.  We now pass to spatial dimension $d\ge2$.  The new feature is the rotation algebra $\mathfrak{so}_d$, which couples the $d$ copies of the same $\mathfrak{sl}_2$-module.  We first prove full rigidity for $m=2\ell\ge3$.  We then treat the low weights separately: the adjoint module $m=2$ is shown to be fully rigid by an invariant-form argument, while for the standard module $m=1$ we determine the complete action on the abelian radical, construct rank-two mixed zero anchors, and close the remaining rotation case by a semisimple-perturbation argument.

Put
\[
S=\mathfrak{sl}_2\oplus\mathfrak{so}_d,
\qquad
V=U_m\otimes W_d,
\qquad W_d=\mathbb C^d,
\]
where $W_d$ is the natural $\mathfrak{so}_d$-module.  Let
\[
\varepsilon_1,\ldots,\varepsilon_d
\]
be its standard orthonormal basis.  We write
\[
V_i=U_m\otimes\mathbb C\varepsilon_i,
\qquad
\pi_i:V\longrightarrow U_m
\]
for the $i$th spatial string and its coordinate projection.  With this notation
\[
\mathfrak g_\ell(d)=S\ltimes V.
\]
For
\[
a\in\mathfrak{sl}_2,
\qquad B\in\mathfrak{so}_d,
\qquad r\in V,
\qquad \lambda\in\mathbb C,
\]
a derivation
\[
D=\operatorname{ad}(a+B+r)+\lambda\delta_0
\]
acts by
\begin{align}
D(x)&=[a+B,x]-\rho(x)r, &&x\in S,\label{eq:D-higher-S}\\
D(v)&=\Bigl((\rho_m(a)+\lambda I_{U_m})\otimes I_{W_d}
      +I_{U_m}\otimes B\Bigr)v, &&v\in V.\label{eq:D-higher-V}
\end{align}
Here $\rho$ denotes the tensor-product action of $S$ on $V$.

\subsection{A reflexivity tool for the radical representation}
The low-weight arguments below require one linear-algebraic fact about the operator space induced on the abelian radical.  We record it here in a form that is valid for every $m\ge1$.

Put
\[
\mathcal A_m=\rho_m(\mathfrak{sl}_2)+\mathbb C I_{U_m},
\qquad
\mathcal B_d=\mathfrak{so}_d+\mathbb C I_{W_d},
\]
and
\[
\mathcal E_{m,d}
=\rho(\mathfrak{sl}_2\oplus\mathfrak{so}_d)+\mathbb C I_V
=\rho_m(\mathfrak{sl}_2)\otimes I+I\otimes\mathfrak{so}_d+\mathbb C I.
\]
For an operator space $\mathcal E\subseteq\operatorname{End}(X)$, write
\[
\operatorname{Ref}(\mathcal E)
=\{T\in\operatorname{End}(X):Tx\in\mathcal E x\text{ for every }x\in X\}.
\]

\begin{lemma}\label{lem:sl2-module-reflexive-2local}
For every $m\ge1$,
\[
\operatorname{Ref}(\mathcal A_m)=\mathcal A_m.
\]
\end{lemma}
\begin{proof}
For $m=1$ one has $\mathcal A_1=\operatorname{End}(U_1)$, so the assertion is immediate.  Assume $m\ge2$ and let $T\in\operatorname{Ref}(\mathcal A_m)$.  Consider the rational normal curve
\[
v(t)=\sum_{k=0}^{m}\binom{m}{k}t^ku_k.
\]
Since $\mathcal A_m v(t)$ is spanned by $v(t)$ and $v'(t)$, there exist scalars depending on $t$ such that
\[
Tv(t)=\alpha(t)v(t)+\beta(t)v'(t).
\]
Write $w_k(t)$ for the $u_k$-coordinate of $Tv(t)$.  Eliminating $\alpha$ and $\beta$ from the first two coordinates gives, for $1\le k\le m$,
\begin{equation}\label{eq:rnc-lowweight-reflexive}
\frac{w_k(t)}{\binom{m}{k}}
=\frac{k}{m}t^{k-1}w_1(t)-(k-1)t^kw_0(t).
\end{equation}
We justify the degree restriction explicitly.  Taking $k=m$ and $k=m-1$ in \eqref{eq:rnc-lowweight-reflexive} and eliminating $w_1$ gives
\[
t^m w_0(t)=t w_{m-1}(t)-(m-1)w_m(t).
\]
The right-hand side has degree at most $m+1$, while the left-hand side is divisible by $t^m$; hence $w_0$ has degree at most one.  The equation for $k=m$ then shows that $w_1$ has degree at most two.  Writing
\[
w_0(t)=a_0+a_1t,
\qquad
w_1(t)=b_0+b_1t+b_2t^2,
\]
the coefficient of $t^{m+1}$ in the $k=m$ equation is $b_2-(m-1)a_1$, so
\[
b_2=(m-1)a_1.
\]
Thus
\[
w_0(t)=a_0+a_1t,
\qquad
w_1(t)=b_0+b_1t+(m-1)a_1t^2.
\]
Substituting these expressions into \eqref{eq:rnc-lowweight-reflexive} gives, for every $0\le k\le m$,
\[
\frac{w_k(t)}{\binom{m}{k}}
=\frac{k}{m}b_0t^{k-1}
 +\left(\frac{k}{m}b_1-(k-1)a_0\right)t^k
 +\frac{m-k}{m}a_1t^{k+1},
\]
where the first or last term is omitted at an endpoint.  If
$Tu_j=\sum_i t_{ij}u_i$, comparison with
$w_k(t)=\sum_j\binom{m}{j}t_{kj}t^j$ shows that the only non-zero entries in the $k$th row are
\[
t_{k,k-1}=\frac{m-k+1}{m}b_0,
\qquad
t_{k,k}=\frac{k}{m}b_1-(k-1)a_0,
\qquad
t_{k,k+1}=\frac{k+1}{m}a_1.
\]
Set
\[
\beta=\frac{a_1}{m},
\qquad
\gamma=\frac{b_0}{m},
\qquad
\delta=\frac12\left(a_0-\frac{b_1}{m}\right),
\qquad
\alpha=a_0-m\delta.
\]
Using
\[
\rho_m(e)u_j=ju_{j-1},\qquad
\rho_m(f)u_j=(m-j)u_{j+1},\qquad
\rho_m(h)u_j=(m-2j)u_j,
\]
the preceding three diagonals are exactly those of
\[
\alpha I+\beta\rho_m(e)+\gamma\rho_m(f)+\delta\rho_m(h).
\]
Denote this operator by $T_0$. The preceding formulas give $(T-T_0)v(t)=0$ for every $t$. Since the vectors $v(t)$ span $U_m$ by the Vandermonde determinant, $T=T_0\in\mathcal A_m$.
\end{proof}

\begin{lemma}\label{lem:so-module-reflexive-2local}
For every $d\ge1$,
\[
\operatorname{Ref}(\mathcal B_d)=\mathcal B_d.
\]
\end{lemma}
\begin{proof}
The case $d=1$ is immediate.  For $d=2$, let $J$ generate $\mathfrak{so}_2$.  Over $\mathbb C$ the two isotropic lines are precisely the two eigenspaces of $J$.  If $w$ lies on either isotropic line, then $(\mathfrak{so}_2+\mathbb CI)w=\mathbb Cw$, so the pointwise condition forces $C$ to preserve both eigenspaces.  Hence $C$ is diagonal in an eigenbasis of $J$, and therefore $C\in\mathbb CI+\mathbb CJ=\mathcal B_2$.

Let $d\ge3$ and suppose $Cw\in(\mathfrak{so}_d+\mathbb CI)w$ for all $w\in W_d$.  If $w$ is isotropic for the standard quadratic form $q(w)=w^{\mathsf T}w$, then
\[
w^{\mathsf T}Cw=0.
\]
Thus the quadratic form determined by the symmetric part of $C$ vanishes on the irreducible quadric $q=0$.  It is therefore a scalar multiple of $q$.  Hence the symmetric part of $C$ is scalar, and $C$ differs from a scalar operator by a skew-symmetric operator.
\end{proof}

\begin{proposition}\label{prop:tensor-module-reflexive-2local}
For $V=U_m\otimes W_d$,
\[
\operatorname{Ref}(\mathcal E_{m,d})=\mathcal E_{m,d}.
\]
\end{proposition}
\begin{proof}
Fix the standard basis $\varepsilon_1,\ldots,\varepsilon_d$ of $W_d$ and write a linear operator $T$ on $V$ as a $d\times d$ block matrix $(T_{ji})$, where each block acts on $U_m$.  Applying the pointwise condition to $u\otimes\varepsilon_i$ gives
\[
T_{ii}u\in\mathcal A_m u,
\qquad
T_{ji}u\in\mathbb Cu\quad(j\ne i).
\]
Lemma~\ref{lem:sl2-module-reflexive-2local} implies $T_{ii}\in\mathcal A_m$, while $T_{ji}=c_{ji}I$ for $j\ne i$.  Applying the condition to $u\otimes(\varepsilon_i+\varepsilon_j)$ shows that $T_{ii}-T_{jj}$ is scalar.  Consequently
\[
T=B\otimes I+I\otimes C
\]
for some $B\in\mathcal A_m$ and $C\in\operatorname{End}(W_d)$.  After subtracting $B\otimes I$, fix non-zero $u\in U_m$ and $w\in W_d$.  The pointwise condition gives
\[
u\otimes Cw=A_u u\otimes w+u\otimes B_w w
\]
for some $A_u\in\mathcal A_m$ and $B_w\in\mathfrak{so}_d$.  Hence
\[
u\otimes(Cw-B_w w)=A_u u\otimes w.
\]
By the elementary rank-one tensor criterion, either both sides are zero or $A_u u$ is proportional to $u$ and $Cw-B_w w$ is proportional to $w$.  In either case
\[
Cw\in(\mathfrak{so}_d+\mathbb CI)w
\qquad(w\in W_d).
\]
Lemma~\ref{lem:so-module-reflexive-2local} gives $C\in\mathcal B_d$, and hence $T\in\mathcal E_{m,d}$.
\end{proof}

\subsection{Normalization on the reductive factor}
We first show that all values on $S$ can be removed by one global derivation.  The proof uses the same two nilpotent anchors $e$ and $f$ as in the one-dimensional case, but the final step also uses them to eliminate the rotation generators.

\begin{proposition}\label{prop:higher-S-normalization}
Let $d\ge2$ and let $\Delta$ be a 2-local derivation on $\mathfrak g_\ell(d)$.  There exists a derivation $D\in\operatorname{Der}\mathfrak g_\ell(d)$ such that, after replacing $\Delta$ by $\Delta-D$,
\[
\Delta(\mathfrak g_\ell(d))\subseteq V,
\qquad
\Delta(x)=0\quad\text{for every }x\in S.
\]
\end{proposition}
\begin{proof}
The ideal $V$ is characteristic.  If $d\ge3$, then
\[
\mathfrak g_\ell(d)/V\cong\mathfrak{sl}_2\oplus\mathfrak{so}_d
\]
is semisimple, so Lemma~\ref{lem:quotient} and the theorem of Ayupov--Kudaybergenov--Rakhimov \cite{AKR} allow us to subtract an inner lift and assume
\[
\Delta(\mathfrak g_\ell(d))\subseteq V.
\]
If $d=2$, the quotient is $\mathfrak{sl}_2\oplus\mathbb Cs$.  Lemma~\ref{lem:d2-reductive-quotient} applies to the induced map and shows that one inner derivation on the $\mathfrak{sl}_2$-factor represents the quotient map at every point $x+\alpha s$.  Subtracting its lift therefore gives
\[
\Delta(\mathfrak g_\ell(2))\subseteq V
\]
with no residual quotient value on mixed $\mathfrak{sl}_2\oplus\mathbb Cs$ elements.

Choose a representing derivation for the pair $(e,f)$ and write it as in \eqref{eq:D-higher-S}--\eqref{eq:D-higher-V}.  Since $\Delta(e),\Delta(f)\in V$, the $S$-components give
\[
[a,e]=[a,f]=0.
\]
Here $a\in\mathfrak{sl}_2$ is the $\mathfrak{sl}_2$-coefficient in \eqref{eq:D-higher-S}; the rotation coefficient $B\in\mathfrak{so}_d$ commutes with both $e$ and $f$.  The intersection of the centralizers of $e$ and $f$ in $\mathfrak{sl}_2$ is zero, so $a=0$.  Hence, for one and the same $r\in V$,
\[
\Delta(e)=-\rho(e)r,
\qquad
\Delta(f)=-\rho(f)r.
\]
Subtracting the inner derivation $\operatorname{ad}r$ preserves the quotient normalization and gives
\[
\Delta(e)=\Delta(f)=0.
\]

We now prove that $\Delta$ vanishes on $\mathfrak{sl}_2$.  The proof of Lemma~\ref{lem:sl2-zero} applies componentwise.  Indeed, for a representing derivation with $D(e)=0$, the $\mathfrak{sl}_2$-part satisfies $a\in\mathbb Ce$ and the radical part lies in
\[
\ker\rho(e)=u_0\otimes W_d,
\]
whereas the rotation part $B$ commutes with $e$ and therefore does not enter $D(x)$ for $x\in\mathfrak{sl}_2$.  Similarly, $D(f)=0$ forces the radical part into $u_m\otimes W_d$.  The support comparisons used in Lemma~\ref{lem:sl2-zero} are unchanged after tensoring with $W_d$.  Consequently
\[
\Delta(x)=0
\qquad(x\in\mathfrak{sl}_2).
\]

It remains to consider $R\in\mathfrak{so}_d$.  Use the pair $(e,R)$.  Since $D(e)=0$, the radical part of the representing derivation belongs to $u_0\otimes W_d$.  Since $\Delta(R)\in V$, the $S$-component of $D(R)$ vanishes, so $[B,R]=0$, and therefore
\[
\Delta(R)=-\rho(R)r\in u_0\otimes R(W_d).
\]
Using the pair $(f,R)$ gives at the same time
\[
\Delta(R)\in u_m\otimes R(W_d).
\]
Because $u_0$ and $u_m$ are linearly independent, these two subspaces have zero intersection.  Hence $\Delta(R)=0$.  Thus $\Delta$ vanishes on all of $S$.
\end{proof}

\subsection{Synchronization of the spatial strings}
We next determine the values on vectors contained in a single spatial string.  For $m\ge3$ the separating vector from Lemma~\ref{lem:separating-vector} synchronizes the $\mathfrak{sl}_2$ and scalar coefficients, while the two-point condition synchronizes the rotation coefficients between different strings.

Set
\[
w_*=u_0+u_1+u_m\in U_m.
\]

\begin{proposition}\label{prop:higher-pure-strings}
Assume $d\ge2$ and $m\ge3$, and let $\Delta$ be a 2-local derivation satisfying
\[
\Delta|_S=0,
\qquad
\Delta(\mathfrak g_\ell(d))\subseteq V.
\]
Then there exist a scalar $c\in\mathbb C$ and a fixed matrix $B_0\in\mathfrak{so}_d$ such that
\begin{equation}\label{eq:pure-string-form}
\Delta(u\otimes\varepsilon_i)
=c\,u\otimes\varepsilon_i+u\otimes B_0\varepsilon_i
\qquad
(u\in U_m,\ 1\le i\le d).
\end{equation}
Moreover, if $d\ge3$, then $B_0=0$.
\end{proposition}
\begin{proof}
Fix $i$.  Applying the pairs $(e,w_*\otimes\varepsilon_i)$ and $(f,w_*\otimes\varepsilon_i)$ and then taking the $i$th spatial coordinate, the skew-symmetry of the rotation matrix removes its diagonal contribution.  Exactly as in the proof of Proposition~\ref{prop:radical-scalar}, we obtain
\[
\pi_i\Delta(w_*\otimes\varepsilon_i)\in
\operatorname{span}\{w_*,\rho_m(e)w_*\}
\cap
\operatorname{span}\{w_*,\rho_m(f)w_*\}
=\mathbb Cw_*.
\]
Write
\[
\pi_i\Delta(w_*\otimes\varepsilon_i)=c_iw_*.
\]
Now let $u\in U_m$ and use the pair
\[
(w_*\otimes\varepsilon_i,\ u\otimes\varepsilon_i).
\]
On $V$ its representing derivation has the operator
\[
T\otimes I+I\otimes B,
\qquad
T=\rho_m(a)+\lambda I_{U_m}.
\]
Taking the $i$th coordinate at the first point gives $Tw_*=c_iw_*$.  Therefore
\[
(T-c_iI)w_*=0.
\]
Lemma~\ref{lem:separating-vector} yields $T=c_iI$, and hence
\begin{equation}\label{eq:diagonal-string}
\pi_i\Delta(u\otimes\varepsilon_i)=c_i u.
\end{equation}

Fix $j\ne i$.  Use the same pair, now taking the $j$th coordinate.  The only contribution is the $(j,i)$ entry of the same skew matrix $B$.  Taking first $u=w_*$ and then an arbitrary $u$ shows that there is a fixed scalar $b_{ji}$ such that
\begin{equation}\label{eq:cross-string}
\pi_j\Delta(u\otimes\varepsilon_i)=b_{ji}u
\qquad(u\in U_m).
\end{equation}
Thus the action of $\Delta$ on the whole string $V_i$ is linear.

Now take the pair
\[
w_*\otimes\varepsilon_i,
\qquad
w_*\otimes\varepsilon_j.
\]
The same operator $T$ occurs in both diagonal coordinates, so \eqref{eq:diagonal-string} gives
\[
c_iw_*=Tw_*=c_jw_*.
\]
Hence all $c_i$ are one common scalar $c$.  The same representing rotation matrix gives
\[
b_{ij}=-b_{ji}.
\]
Therefore the numbers $b_{ji}$ are the matrix entries of one fixed $B_0\in\mathfrak{so}_d$, and \eqref{eq:pure-string-form} follows.

Assume now that $d\ge3$.  Fix $i$ and choose two distinct indices $j,k\ne i$.  Let $R_{ij}\in\mathfrak{so}_d$ be the standard rotation in the $(i,j)$-plane.  Use the pair
\[
(R_{ij},\ u\otimes\varepsilon_i).
\]
Since $\Delta(R_{ij})=0$, the $\mathfrak{so}_d$-part $B$ of the representing derivation satisfies
\[
[B,R_{ij}]=0.
\]
The centralizer of $R_{ij}$ preserves the plane $\operatorname{span}\{\varepsilon_i,\varepsilon_j\}$ and its orthogonal complement; because $B$ is skew-symmetric,
\[
B\varepsilon_i\in\mathbb C\varepsilon_j.
\]
Comparing the spatial components in \eqref{eq:pure-string-form} gives
\[
B_0\varepsilon_i\in\mathbb C\varepsilon_j.
\]
Repeating the argument with $R_{ik}$ gives
\[
B_0\varepsilon_i\in\mathbb C\varepsilon_k.
\]
Thus $B_0\varepsilon_i=0$.  Since $i$ was arbitrary, $B_0=0$.
\end{proof}

\begin{corollary}\label{cor:higher-coordinate-normalization}
Assume $d\ge2$ and $m\ge3$.  After subtracting a global derivation, a 2-local derivation may be normalized so that
\begin{equation}\label{eq:higher-coordinate-zero}
\Delta|_S=0,
\qquad
\Delta(u\otimes\varepsilon_i)=0
\quad(u\in U_m,\ 1\le i\le d),
\qquad
\Delta(\mathfrak g_\ell(d))\subseteq V.
\end{equation}
\end{corollary}
\begin{proof}
Apply Proposition~\ref{prop:higher-S-normalization} and then Proposition~\ref{prop:higher-pure-strings}.  Subtract $c\delta_0$.  If $d\ge3$, the rotation term is already zero.  If $d=2$, then $B_0\in\mathfrak{so}_2$ and the inner derivation $\operatorname{ad}B_0$ vanishes on $S=\mathfrak{sl}_2\oplus\mathfrak{so}_2$ because $\mathfrak{so}_2$ is abelian and commutes with $\mathfrak{sl}_2$.  Subtracting it preserves $\Delta|_S=0$ and yields \eqref{eq:higher-coordinate-zero}.
\end{proof}

\subsection{Rigidity on the whole abelian radical}
The previous corollary only gives zero values on the coordinate strings.  The next theorem removes the nonlinearity on arbitrary sums of spatial strings.  The key point is that one zero anchor $w_*\otimes\varepsilon_i$ forces the full $\mathfrak{sl}_2$-scalar operator to vanish and simultaneously kills the $i$th row and column of the representing rotation.

\begin{theorem}\label{thm:higher-radical-zero}
Assume $d\ge2$ and $m=2\ell\ge3$.  Let $\Delta$ satisfy the normalization \eqref{eq:higher-coordinate-zero}.  Then
\[
\Delta(v)=0
\qquad\text{for every }v\in V.
\]
Consequently, every 2-local derivation on $\mathfrak g_\ell(d)$ differs from a global derivation by a 2-local derivation that vanishes on
\[
S\cup V.
\]
\end{theorem}
\begin{proof}
Fix $v\in V$ and an index $i$.  Put
\[
p_i=w_*\otimes\varepsilon_i.
\]
By \eqref{eq:higher-coordinate-zero}, $\Delta(p_i)=0$.  Use a representing derivation for the pair $(p_i,v)$.  On $V$ it has the form
\[
A=T\otimes I+I\otimes B,
\qquad
T=\rho_m(a)+\lambda I,
\qquad
B\in\mathfrak{so}_d.
\]
The equality $A p_i=0$ reads
\begin{equation}\label{eq:rank-one-anchor}
Tw_*\otimes\varepsilon_i+w_*\otimes B\varepsilon_i=0.
\end{equation}
For non-zero vectors $u,w$, an identity
\[
u'\otimes w+u\otimes w'=0
\]
forces $u'=\alpha u$ and $w'=-\alpha w$ for some scalar $\alpha$.  Applying this elementary rank-one tensor fact to \eqref{eq:rank-one-anchor}, we get
\[
Tw_*=\alpha w_*,
\qquad
B\varepsilon_i=-\alpha\varepsilon_i.
\]
Since $(T-\alpha I)w_*=0$, Lemma~\ref{lem:separating-vector} gives
\[
T=\alpha I.
\]
But $B$ is skew-symmetric and $\langle\varepsilon_i,\varepsilon_i\rangle=1$, so
\[
0=\langle B\varepsilon_i,\varepsilon_i\rangle=-\alpha.
\]
Thus
\[
T=0,
\qquad
B\varepsilon_i=0.
\]
Skew-symmetry then also gives
\[
\langle B\varepsilon_j,\varepsilon_i\rangle=0
\qquad(1\le j\le d),
\]
so the $i$th spatial coordinate of $(I\otimes B)v$ is zero.  Therefore
\[
\pi_i\Delta(v)=0.
\]
Since this holds for every $i$, we conclude $\Delta(v)=0$.
\end{proof}

\begin{corollary}\label{cor:higher-basis-rigidity}
Let $d\ge2$ and $2\ell\ge3$.  For every 2-local derivation $\Delta$ on $\mathfrak g_\ell(d)$ there exists $D\in\operatorname{Der}\mathfrak g_\ell(d)$ such that
\[
(\Delta-D)(x)=0
\qquad\text{for every }x\in S\cup V.
\]
\end{corollary}

\subsection{Mixed elements and full higher-dimensional rigidity}
The preceding result determines a normalized 2-local derivation on each of the two summands separately. Since a 2-local derivation is not assumed to be additive, one must still control values at $x+v$ with both components non-zero. We now construct a family of zero mixed anchors that eliminates the remaining freedom.

Set, as before,
\[
w_*=u_0+u_1+u_m,
\qquad
p_i=w_*\otimes\varepsilon_i.
\]
The next lemma is the higher-dimensional analogue of the two zero anchors used in the one-string proof.

\begin{lemma}\label{lem:higher-zero-anchors}
Assume $d\ge2$ and $m\ge3$, and let $\Delta$ satisfy
\[
\Delta|_S=0,
\qquad
\Delta|_V=0,
\qquad
\Delta(\mathfrak g_\ell(d))\subseteq V.
\]
Then, for every $1\le i\le d$,
\[
\Delta(e+p_i)=0,
\qquad
\Delta(f+p_i)=0.
\]
\end{lemma}
\begin{proof}
Put $z_i=e+p_i$. Use first the pair $(p_i,z_i)$. Since $\Delta(p_i)=0$, the restriction to $V$ of a representing derivation has the form
\[
A=T\otimes I+I\otimes B,
\qquad
T=\rho_m(a)+\lambda I,
\qquad B\in\mathfrak{so}_d,
\]
and satisfies $Ap_i=0$. Thus
\[
Tw_*\otimes\varepsilon_i+w_*\otimes B\varepsilon_i=0.
\]
By the rank-one tensor argument used in Theorem~\ref{thm:higher-radical-zero}, there is $\alpha\in\mathbb C$ such that
\[
Tw_*=\alpha w_*,
\qquad
B\varepsilon_i=-\alpha\varepsilon_i.
\]
Lemma~\ref{lem:separating-vector} gives $T=\alpha I$. Since $B$ is skew-symmetric and $\varepsilon_i$ is non-isotropic, $\alpha=0$. Hence
\[
T=0,
\qquad B\varepsilon_i=0.
\]
In particular $a=0$ and $\lambda=0$, and therefore
\[
\Delta(z_i)=-\rho(e)r\in\operatorname{Im}\rho(e).
\]

Now use the pair $(f,z_i)$. The equality $D(f)=0$ gives
\[
a\in\mathbb C f,
\qquad
r\in\ker\rho(f)=u_m\otimes W_d.
\]
Because $\Delta(z_i)$ has no $S$-component, the $S$-component of $D(z_i)$ forces $a=0$. Write $r=u_m\otimes q$. The $u_m$-coefficient of the $V$-part of $D(z_i)$ is
\[
w_*\otimes(\lambda\varepsilon_i+B\varepsilon_i).
\]
But $\Delta(z_i)\in\operatorname{Im}\rho(e)$ and this image has no $u_m$-component. Hence
\[
\lambda\varepsilon_i+B\varepsilon_i=0.
\]
Taking the scalar product with $\varepsilon_i$ gives $\lambda=0$, and then $B\varepsilon_i=0$. Consequently
\[
\Delta(z_i)\in u_{m-1}\otimes W_d.
\]

Finally use the pair $(h,z_i)$. We obtain
\[
a\in\mathbb C h,
\qquad
r\in\ker\rho(h).
\]
Again the $S$-component of $D(z_i)$ forces $a=0$, and the $u_m$-coefficient gives
\[
\lambda=0,
\qquad B\varepsilon_i=0.
\]
If $m$ is odd, then $\ker\rho(h)=0$, so $\Delta(z_i)=0$. If $m$ is even, then $m\ge4$ and
\[
\ker\rho(h)=u_{m/2}\otimes W_d,
\]
whence
\[
\Delta(z_i)\in u_{m/2-1}\otimes W_d.
\]
Its intersection with $u_{m-1}\otimes W_d$ is zero. Thus $\Delta(e+p_i)=0$ in all cases. The proof of $\Delta(f+p_i)=0$ is symmetric.
\end{proof}

\begin{lemma}\label{lem:higher-conjugate-anchor}
Let $G\in\operatorname{Aut}(\mathfrak{sl}_2)$ be inner and let the same symbol denote its natural extension to an automorphism of $\mathfrak g_\ell(d)$ acting trivially on $\mathfrak{so}_d$ and as the corresponding $U_m$-operator on the first tensor factor of $V$. Under the assumptions of Lemma~\ref{lem:higher-zero-anchors},
\[
\Delta\bigl(G(e+p_i)\bigr)=0,
\qquad
\Delta\bigl(G(f+p_i)\bigr)=0
\]
for every $i$.
\end{lemma}
\begin{proof}
The conjugate map $G^{-1}\Delta G$ is again a 2-local derivation and vanishes on $S$ and $V$ separately. Apply Lemma~\ref{lem:higher-zero-anchors} to this conjugate map and then conjugate back.
\end{proof}

We shall also need one elementary fact about the $\mathfrak{sl}_2$-module $U_m$.

\begin{lemma}\label{lem:nilpotent-images-intersection}
Let
\[
s_t=\operatorname{Ad}(\exp(tf))e,
\qquad t\in\mathbb C.
\]
Then
\[
\bigcap_{t\in\mathbb C}\operatorname{Im}\rho_m(s_t)=\{0\}.
\]
Consequently,
\[
\bigcap_{t\in\mathbb C}\operatorname{Im}\rho(s_t)=\{0\}
\qquad\text{on }V=U_m\otimes W_d.
\]
\end{lemma}
\begin{proof}
Let $G_t=\exp(t\rho_m(f))$. Then
\[
\operatorname{Im}\rho_m(s_t)=G_t\operatorname{Im}\rho_m(e).
\]
A linear functional annihilating this hyperplane is
\[
\varphi_t(u_j)=(-t)^{m-j},
\qquad 0\le j\le m,
\]
up to a non-zero scalar factor. If $u=\sum_{j=0}^m c_j u_j$ belongs to every $\operatorname{Im}\rho_m(s_t)$, then
\[
0=\varphi_t(u)=\sum_{j=0}^m c_j(-t)^{m-j}
\qquad\text{for every }t\in\mathbb C.
\]
Hence all $c_j$ vanish. The tensor-product statement follows by applying the same argument to each spatial coordinate.
\end{proof}

\begin{theorem}\label{thm:higher-full-mge3}
Let $d\ge2$ and $m=2\ell\ge3$. Then every 2-local derivation on
\[
\mathfrak g_\ell(d)=
(\mathfrak{sl}_2\oplus\mathfrak{so}_d)\ltimes(U_m\otimes\mathbb C^d)
\]
is a derivation. Equivalently,
\[
\TwoLocDer\mathfrak g_\ell(d)=\Der\mathfrak g_\ell(d).
\]
\end{theorem}
\begin{proof}
By Corollary~\ref{cor:higher-basis-rigidity}, after subtracting a global derivation we may assume
\begin{equation}\label{eq:higher-full-normalized}
\Delta|_S=0,
\qquad
\Delta|_V=0,
\qquad
\Delta(\mathfrak g_\ell(d))\subseteq V.
\end{equation}
We prove that this normalized map is zero.

Let
\[
y=x+v,
\qquad
x=x_0+R,
\quad x_0\in\mathfrak{sl}_2,
\quad R\in\mathfrak{so}_d,
\quad v\in V.
\]
The pair $(x,y)$ and $\Delta(x)=0$ show that $\Delta(y)\in V$.

Fix $i$. Suppose first that $x_0\notin\mathbb C e$. Use the zero anchor
\[
z_i^e=e+p_i.
\]
For a representing derivation of $(z_i^e,y)$, the equality $D(z_i^e)=0$ gives
\[
a=\alpha e.
\]
Looking at the $u_m$-coefficient in the radical equation yields
\[
\lambda\varepsilon_i+B\varepsilon_i=0.
\]
Skew-symmetry gives $\lambda=0$ and $B\varepsilon_i=0$. The remaining radical equation then gives
\[
r=\alpha p_i+u_0\otimes q
\]
for some $q\in W_d$. Since $D(y)=\Delta(y)$ has no $S$-component,
\[
[a,x_0]=0,
\qquad [B,R]=0.
\]
Because $x_0\notin\mathbb C e$, we get $\alpha=0$. Hence
\[
r=u_0\otimes q.
\]
Moreover, $B\varepsilon_i=0$ and skew-symmetry imply that the $i$th row of $B$ is zero. Taking the $i$th spatial coordinate of $D(y)$ therefore gives
\begin{equation}\label{eq:higher-e-support}
\pi_i\Delta(y)
\in
\operatorname{span}\{u_0,\rho_m(x_0)u_0\}
\subseteq\operatorname{span}\{u_0,u_1\}.
\end{equation}
Similarly, if $x_0\notin\mathbb C f$, the zero anchor $f+p_i$ gives
\begin{equation}\label{eq:higher-f-support}
\pi_i\Delta(y)
\in
\operatorname{span}\{u_m,\rho_m(x_0)u_m\}
\subseteq\operatorname{span}\{u_{m-1},u_m\}.
\end{equation}
If
\[
x_0\notin\mathbb C e\cup\mathbb C f,
\]
then the two support spaces in \eqref{eq:higher-e-support} and \eqref{eq:higher-f-support} have zero intersection because $m\ge3$. Thus
\[
\pi_i\Delta(y)=0.
\]
Since $i$ is arbitrary, $\Delta(y)=0$.

We next treat $x_0=Ae$ with $A\ne0$. The anchor $f+p_i$ gives
\begin{equation}\label{eq:higher-e-case-top}
\pi_i\Delta(y)\in\operatorname{span}\{u_{m-1},u_m\}.
\end{equation}
Choose $t\ne0$ and put
\[
G_t=\exp(t\rho_m(f)),
\qquad
s_t=\operatorname{Ad}(\exp(tf))e,
\qquad
r_t=G_t(w_*),
\]
\[
z_{t,i}=s_t+r_t\otimes\varepsilon_i.
\]
By Lemma~\ref{lem:higher-conjugate-anchor}, $\Delta(z_{t,i})=0$. Apply the two-local condition to $(z_{t,i},y)$. Since $s_t\notin\mathbb C e$ for $t\ne0$, while $x_0=Ae$ with $A\ne0$, the $\mathfrak{sl}_2$-part of the representing derivation is forced to be zero. To read the radical equation without changing weight coordinates, conjugate this representing derivation by the inverse of the same inner automorphism. It then annihilates the basic anchor $e+p_i$. The calculation used above for an arbitrary derivation annihilating $e+p_i$ gives
\[
\lambda=0,
\qquad B\varepsilon_i=0,
\qquad
r\in G_t(u_0)\otimes W_d.
\]
Consequently
\begin{equation}\label{eq:higher-e-case-conj}
\pi_i\Delta(y)
\in
\operatorname{span}\{G_t(u_0),A\rho_m(e)G_t(u_0)\}.
\end{equation}
Now
\[
G_t(u_0)=\sum_{k=0}^m\binom{m}{k}t^k u_k,
\]
and the coefficients of $u_0,u_1$ in a vector
\[
P G_t(u_0)+Q A\rho_m(e)G_t(u_0)
\]
are
\[
P+QAm t,
\qquad
Pmt+QA m(m-1)t^2.
\]
If both coefficients vanish and $A t\ne0$, then $P=Q=0$. Hence the space in \eqref{eq:higher-e-case-conj} has zero intersection with $\operatorname{span}\{u_{m-1},u_m\}$ for $m\ge3$. Together with \eqref{eq:higher-e-case-top} this gives
\[
\pi_i\Delta(y)=0.
\]
Again $i$ is arbitrary, so $\Delta(y)=0$. The case
\[
x_0\in\mathbb C f\setminus\{0\}
\]
is symmetric, using $e+p_i$ and the conjugates of $f+p_i$ under $\exp(t\operatorname{ad}e)$.

It remains to consider $x_0=0$, so
\[
y=R+v.
\]
For every $t\in\mathbb C$, set
\[
y_t=s_t+R+v.
\]
The $\mathfrak{sl}_2$-component of $y_t$ is non-zero, so the cases already proved give
\[
\Delta(y_t)=0.
\]
Apply the 2-local property to the pair $(y,y_t)$. If $D_t$ is the representing derivation, then
\[
D_t(s_t)=D_t(y_t)-D_t(y)=-\Delta(y).
\]
The right-hand side lies in $V$. Therefore the $S$-component of $D_t(s_t)$ vanishes and its radical component has the form
\[
-\rho(s_t)r_t'.
\]
Thus
\[
\Delta(y)\in\operatorname{Im}\rho(s_t)
\qquad\text{for every }t\in\mathbb C.
\]
Lemma~\ref{lem:nilpotent-images-intersection} yields $\Delta(y)=0$.

Hence the normalized 2-local derivation vanishes on every element of $\mathfrak g_\ell(d)$. Undoing the initial subtraction proves the theorem.
\end{proof}

\begin{corollary}\label{cor:noncentral-mge3-all-d}
For every $d\ge1$ and every positive half-integer $\ell$ with $2\ell\ge3$,
\[
\TwoLocDer\mathfrak g_\ell(d)=\Der\mathfrak g_\ell(d).
\]
In addition, the equality holds for $d=1$ and $2\ell=1,2$ by Corollary~\ref{cor:one-dimensional-all}.
\end{corollary}

\subsection{The adjoint module \texorpdfstring{$m=2$}{m=2} in higher spatial dimension}
We now remove the restriction $m\ge3$ in the adjoint-module case.  The separating-vector argument is replaced by the invariant symmetric form on $U_2$ introduced in \eqref{eq:B-U2}.  Its tensor product with the standard form on $W_d$ controls the nonlinearity of the restriction of a 2-local derivation to the radical.

Define the nondegenerate symmetric bilinear form
\begin{equation}\label{eq:Q-m2-higher}
Q(u\otimes w,u'\otimes w')=B(u,u')\,w^{\mathsf T}w'
\qquad(u,u'\in U_2,\ w,w'\in W_d).
\end{equation}
By \eqref{eq:B-invariance-U2}, every operator in $\rho(\mathfrak{sl}_2\oplus\mathfrak{so}_d)$ is $Q$-skew-symmetric.

\begin{proposition}\label{prop:higher-m2-radical-normalization}
Let $d\ge2$, $m=2$, and let $\Delta$ be a 2-local derivation satisfying
\[
\Delta|_S=0,
\qquad
\Delta(\mathfrak g_1(d))\subseteq V.
\]
Then there exists a derivation $D$ which vanishes on $S$ such that, after replacing $\Delta$ by $\Delta-D$,
\[
\Delta|_V=0.
\]
\end{proposition}
\begin{proof}
For $v,w\in V$, a derivation representing the pair $(v,w)$ restricts to $V$ as
\[
A=K+\lambda I,
\qquad
K\in\rho(S),
\]
where $K$ is $Q$-skew-symmetric.  If $Q(v,v)\ne0$, then
\[
\frac{Q(\Delta(v),v)}{Q(v,v)}=\lambda.
\]
Using one representing derivation for two non-isotropic vectors shows that the displayed quotient has one common value $c$ on the non-isotropic cone complement.

Put
\[
F(v)=\Delta(v)-cv.
\]
If $y$ is non-isotropic, then for every $x\in V$ the representing derivation of $(x,y)$ has scalar coefficient $c$, and therefore
\begin{equation}\label{eq:Q-skew-two-local}
Q(F(x),y)+Q(x,F(y))=0.
\end{equation}
For arbitrary $x_1,x_2\in V$ and every non-isotropic $y$, relation \eqref{eq:Q-skew-two-local} gives
\[
Q(F(x_1+x_2)-F(x_1)-F(x_2),y)=0.
\]
The non-isotropic vectors span $V$ and $Q$ is nondegenerate, hence $F$ is additive.  Lemma~\ref{lem:homogeneous} gives complex homogeneity, so $F$, and therefore $\Delta|_V$, is linear.

For each $v\in V$, the two-local condition gives
\[
\Delta(v)\in\mathcal E_{2,d}v.
\]
Proposition~\ref{prop:tensor-module-reflexive-2local} now yields
\begin{equation}\label{eq:m2-linear-radical-operator}
\Delta|_V=\rho(b)+c_0I
\end{equation}
for some $b=b_1+b_2\in\mathfrak{sl}_2\oplus\mathfrak{so}_d$ and $c_0\in\mathbb C$.

We show that $b$ is central in $S$.  Fix $x\in S$ and $v\in V$, and use a representing derivation of the pair $(x,v)$.  Since $\Delta(x)=0$, its reductive part belongs to $C_S(x)$ and its radical part is annihilated by $\rho(x)$.  Consequently \eqref{eq:m2-linear-radical-operator} implies
\begin{equation}\label{eq:m2-centralizer-pointwise}
(\rho(b)+c_0I)v
\in
(\rho(C_S(x))+\mathbb CI)v
\qquad(x\in S,\ v\in V).
\end{equation}
Take first a non-zero nilpotent $x\in\mathfrak{sl}_2$.  Applying a spatial functional to \eqref{eq:m2-centralizer-pointwise} on pure tensors shows
\[
\rho_2(b_1)u\in\mathbb Cu+\mathbb C\rho_2(x)u
\qquad(u\in U_2).
\]
For $x=e$ this excludes the $f$-coefficient of $b_1$.  Repeating the same argument for all inner conjugates of $e$ places $b_1$ in every Borel subalgebra of $\mathfrak{sl}_2$.  Their intersection is zero, hence $b_1=0$.

Now take $x=R_{ij}$, the standard rotation in the $(i,j)$-plane.  If $w\in\operatorname{span}\{\varepsilon_i,\varepsilon_j\}$, relation \eqref{eq:m2-centralizer-pointwise}, followed by a functional on the $U_2$-factor, shows that
\[
b_2w\in
C_{\mathfrak{so}_d}(R_{ij})w+\mathbb Cw
\subseteq\operatorname{span}\{\varepsilon_i,\varepsilon_j\}.
\]
Thus $b_2$ preserves every coordinate two-plane.  For $d\ge3$ this forces $b_2=0$ by skew-symmetry. Indeed, $b_2\varepsilon_i$ belongs to every plane $\operatorname{span}\{\varepsilon_i,\varepsilon_j\}$ with $j\ne i$, so it is a multiple of $\varepsilon_i$, and skew-symmetry makes that multiple zero.  For $d=2$, $b_2\in\mathfrak{so}_2=Z(S)$.  Hence in every case $b\in Z(S)$.

The derivation
\[
D=\operatorname{ad}b+c_0\delta_0
\]
vanishes on $S$ and agrees with $\Delta$ on $V$.  Subtracting it proves the proposition.
\end{proof}

\begin{lemma}\label{lem:higher-m2-zero-anchors}
Assume $d\ge2$, $m=2$, and
\[
\Delta|_S=0,
\qquad
\Delta|_V=0,
\qquad
\Delta(\mathfrak g_1(d))\subseteq V.
\]
Put
\[
w=u_0+u_1+u_2,
\qquad
p_i=w\otimes\varepsilon_i.
\]
Then, for every $i$,
\[
\Delta(e+p_i)=\Delta(f+p_i)=\Delta(h+p_i)=0.
\]
\end{lemma}
\begin{proof}
We prove the assertion for $e+p_i$; the $f$-case is symmetric.  Put $z=e+p_i$.  The pair $(e,z)$ first shows that $\Delta(z)\in V$.

Use the pair $(p_i,z)$.  Since $D(p_i)=0$ and the $S$-component of $D(z)$ is zero, the $\mathfrak{sl}_2$-part of the representing derivation lies in $\mathbb Ce$.  Thus on $V$ it has the form
\[
T\otimes I+I\otimes B_0,
\qquad
T=\alpha\rho_2(e)+\lambda I.
\]
The equality at $p_i$ is
\[
Tw\otimes\varepsilon_i+w\otimes B_0\varepsilon_i=0.
\]
The rank-one tensor argument gives $Tw=\beta w$ and $B_0\varepsilon_i=-\beta\varepsilon_i$.  Since $B_0$ is skew-symmetric, $\beta=0$.  As
\[
\rho_2(e)w=u_0+2u_1,
\]
the equality $Tw=0$ gives $\alpha=\lambda=0$, and hence $B_0\varepsilon_i=0$.  Therefore
\[
\Delta(z)\in\operatorname{Im}\rho_2(e)\otimes W_d
=\operatorname{span}\{u_0,u_1\}\otimes W_d.
\]

Next use the pair $(f,z)$.  Write the representing derivation as
\[
D=\operatorname{ad}(a+r)+\lambda\delta_0,
\qquad
a=a_1+B,\quad a_1\in\mathfrak{sl}_2,\quad B\in\mathfrak{so}_d.
\]
The equality $D(f)=0$ gives
\[
a_1=\beta f,
\qquad
r\in\ker\rho_2(f)\otimes W_d=u_2\otimes W_d.
\]
Since $\Delta(z)$ has no $S$-component, the $S$-component
$[\beta f,e]=-\beta h$ of $D(z)$ forces $\beta=0$.  Put
$r=u_2\otimes q$.  The radical part is then
\[
D(z)
=-2u_1\otimes q
+(u_0+u_1+u_2)\otimes(\lambda\varepsilon_i+B\varepsilon_i).
\]
The preceding inclusion
\[
\Delta(z)\in\operatorname{span}\{u_0,u_1\}\otimes W_d
\]
has no $u_2$-component, so
\[
\lambda\varepsilon_i+B\varepsilon_i=0.
\]
Taking the scalar product with $\varepsilon_i$ and using skew-symmetry gives
$\lambda=0$ and then $B\varepsilon_i=0$.  Consequently
\[
\Delta(z)\in u_1\otimes W_d.
\]

Finally use the pair $(h,z)$.  The equality $D(h)=0$ yields
\[
a_1=\gamma h,
\qquad
r\in\ker\rho_2(h)\otimes W_d=u_1\otimes W_d.
\]
The $S$-component of $D(z)$ is $[\gamma h,e]=2\gamma e$, hence
$\gamma=0$.  Writing $r=u_1\otimes q$, we obtain
\[
D(z)
=-u_0\otimes q
+(u_0+u_1+u_2)\otimes(\lambda\varepsilon_i+B\varepsilon_i).
\]
Since the same vector is already known to lie in $u_1\otimes W_d$, its
$u_2$-component gives
\[
\lambda\varepsilon_i+B\varepsilon_i=0,
\]
and therefore $\lambda=0$, $B\varepsilon_i=0$.  Thus
\[
\Delta(z)\in u_0\otimes W_d.
\]
The two one-row restrictions have zero intersection, proving
$\Delta(e+p_i)=0$.

For $z=h+p_i$, use first the pair $(p_i,z)$.  Since the value has no
$S$-component, the $\mathfrak{sl}_2$-part centralizes $h$, so it is
$\gamma h$.  The equality at $p_i$ is the rank-one identity with
\[
T=\gamma\rho_2(h)+\lambda I.
\]
It forces $Tw=0$ and $B\varepsilon_i=0$; because
\[
\rho_2(h)w=2u_0-2u_2
\]
is not proportional to $w$, we get $\gamma=\lambda=0$.  Hence
\[
\Delta(z)\in\operatorname{Im}\rho_2(h)\otimes W_d
=\operatorname{span}\{u_0,u_2\}\otimes W_d.
\]
For the pair $(e,z)$, the equations $D(e)=0$ and the vanishing of the
$S$-component of $D(z)$ force the $\mathfrak{sl}_2$-coefficient to be
zero and the radical parameter to lie in $u_0\otimes W_d$.  Comparing
the $u_1$-component with the preceding support again gives
$\lambda\varepsilon_i+B\varepsilon_i=0$, hence $\lambda=0$ and
$B\varepsilon_i=0$.  Therefore
\[
\Delta(z)\in u_0\otimes W_d.
\]
The pair $(f,z)$ symmetrically gives
\[
\Delta(z)\in u_2\otimes W_d.
\]
Their intersection is zero, so $\Delta(h+p_i)=0$.
\end{proof}

\begin{theorem}\label{thm:higher-m2-full}
Let $d\ge2$.  Every 2-local derivation on
\[
\mathfrak g_1(d)
=(\mathfrak{sl}_2\oplus\mathfrak{so}_d)\ltimes(U_2\otimes\mathbb C^d)
\]
is a derivation.
\end{theorem}
\begin{proof}
By Proposition~\ref{prop:higher-S-normalization} and Proposition~\ref{prop:higher-m2-radical-normalization}, after subtracting a global derivation we may assume
\begin{equation}\label{eq:higher-m2-normalized}
\Delta|_S=0,
\qquad
\Delta|_V=0,
\qquad
\Delta(\mathfrak g_1(d))\subseteq V.
\end{equation}
We prove that this normalized map is zero.

Let
\[
y=x_0+R+v,
\qquad
x_0=Ae+Bf+Ch\in\mathfrak{sl}_2,
\quad R\in\mathfrak{so}_d,
\quad v\in V.
\]
The pair $(x_0+R,y)$ shows that $\Delta(y)\in V$.  Fix an index $i$.

If $x_0\notin\mathbb Ce$, the zero anchor $e+p_i$ gives, by the same coefficient comparison as in \eqref{eq:higher-e-support},
\begin{equation}\label{eq:m2-e-support-higher}
\pi_i\Delta(y)
\in\operatorname{span}\{u_0,\rho_2(x_0)u_0\}.
\end{equation}
If $x_0\notin\mathbb Cf$, the anchor $f+p_i$ gives
\begin{equation}\label{eq:m2-f-support-higher}
\pi_i\Delta(y)
\in\operatorname{span}\{u_2,\rho_2(x_0)u_2\}.
\end{equation}
If $x_0\notin\mathbb Ch$, the anchor $h+p_i$ gives
\begin{equation}\label{eq:m2-h-support-higher}
\pi_i\Delta(y)
\in\operatorname{span}\{u_1,\rho_2(x_0)u_1\}.
\end{equation}
These inclusions already imply $\pi_i\Delta(y)=0$ unless either $AB\ne0$ or $x_0$ is a non-zero multiple of $e$ or $f$.  In each exceptional situation the three basic inclusions reduce the possible value to
\[
\pi_i\Delta(y)\in\mathbb Cu_1.
\]

It remains to eliminate this line.  For $t\in\mathbb C$ put
\[
G_t=\exp(t\rho_2(f)),
\qquad
s_t=\operatorname{Ad}(\exp(tf))e,
\qquad
r_t=G_t(w).
\]
Conjugating the normalized map by the corresponding inner automorphism and applying Lemma~\ref{lem:higher-m2-zero-anchors} shows that
\[
\Delta(s_t+r_t\otimes\varepsilon_i)=0.
\]
Choose $t\ne0$ such that $x_0\notin\mathbb Cs_t$.  Repeating the endpoint calculation with this conjugate zero anchor gives
\begin{equation}\label{eq:m2-conjugate-support-higher}
\pi_i\Delta(y)
\in
\operatorname{span}\{G_tu_0,\rho_2(x_0)G_tu_0\}.
\end{equation}
Since
\[
G_tu_0=u_0+2tu_1+t^2u_2,
\]
a direct determinant computation gives
\begin{equation}\label{eq:m2-determinant-higher}
\det\bigl(G_tu_0,\rho_2(x_0)G_tu_0,u_1\bigr)
=2t\bigl(At^2+2Ct-B\bigr).
\end{equation}
In each of the exceptional cases above the polynomial on the right is not identically zero.  We may therefore choose $t$ so that it is non-zero and $x_0\notin\mathbb Cs_t$.  Then the plane in \eqref{eq:m2-conjugate-support-higher} does not contain $u_1$, and hence $\pi_i\Delta(y)=0$.

Thus $\Delta(y)=0$ whenever $x_0\ne0$.  It remains to take $x_0=0$, so $y=R+v$.  For every $t$ set
\[
y_t=s_t+R+v.
\]
The case already proved gives $\Delta(y_t)=0$.  Using one representing derivation for $(y,y_t)$ yields
\[
\Delta(y)\in\operatorname{Im}\rho_2(s_t)\otimes W_d
\qquad\text{for every }t.
\]
Lemma~\ref{lem:nilpotent-images-intersection} gives $\Delta(y)=0$.  Therefore the normalized map is zero and the theorem follows.
\end{proof}

\begin{corollary}\label{cor:noncentral-all-mge2}
For every $d\ge1$ and every positive half-integer $\ell$ with $2\ell\ge2$,
\[
\TwoLocDer\mathfrak g_\ell(d)=\Der\mathfrak g_\ell(d).
\]
\end{corollary}
\begin{proof}
For $d=1$ this is Corollary~\ref{cor:one-dimensional-all}.  For $d\ge2$ and $2\ell\ge3$ use Theorem~\ref{thm:higher-full-mge3}, while $2\ell=2$ is Theorem~\ref{thm:higher-m2-full}.
\end{proof}

\subsection{The standard module \texorpdfstring{$m=1$}{m=1}: rigidity on the radical}
The standard module behaves differently because $\mathcal A_1=\operatorname{End}(U_1)$; hence no single vector can separate the $\mathfrak{sl}_2$ and scalar coefficients.  Nevertheless the two-point condition still determines the restriction to the entire abelian radical.

\begin{lemma}\label{lem:2local-conformal-orthogonal}
Let $F:W_d\to W_d$ be a 2-local map with respect to
\[
\mathfrak{so}_d+\mathbb CI,
\]
that is, for every $w_1,w_2$ there exists $B+\lambda I$ with $B\in\mathfrak{so}_d$ agreeing with $F$ at both points.  Then $F$ is linear and
\[
F\in\mathfrak{so}_d+\mathbb CI.
\]
\end{lemma}
\begin{proof}
Homogeneity follows exactly as in Lemma~\ref{lem:homogeneous}.  Let $\langle\cdot,\cdot\rangle$ be the standard symmetric bilinear form.  For a non-isotropic $w$ put
\[
c(w)=\frac{\langle F(w),w\rangle}{\langle w,w\rangle}.
\]
A single representing operator for two non-isotropic vectors shows that $c(w)$ is one common scalar $c$.  Put $G=F-cI$.  If $y$ is non-isotropic, then the representing operator for $(x,y)$ has scalar coefficient $c$, and hence
\[
\langle G(x),y\rangle+\langle x,G(y)\rangle=0
\qquad(x\in W_d).
\]
Testing $G(x_1+x_2)-G(x_1)-G(x_2)$ against all non-isotropic $y$ proves additivity.  Together with homogeneity this gives linearity.  The displayed identity then shows that $G$ is skew-symmetric, proving the assertion.
\end{proof}

\begin{proposition}\label{prop:higher-m1-pure-tensors}
Let $d\ge2$, $m=1$, and let $\Delta$ be a 2-local derivation satisfying
\[
\Delta|_S=0,
\qquad
\Delta(\mathfrak g_{1/2}(d))\subseteq V.
\]
Then there exist $B_0\in\mathfrak{so}_d$ and a scalar function $c$ on the projective line $\mathbb P(U_1)$ such that
\begin{equation}\label{eq:m1-pure-tensor-form-higher}
\Delta(u\otimes w)
=u\otimes\bigl(B_0w+c([u])w\bigr)
\qquad(0\ne u\in U_1,\ w\in W_d).
\end{equation}
Moreover, $B_0=0$ for $d\ge3$, while for $d=2$ the term $I\otimes B_0$ is the restriction of an inner derivation which vanishes on $S$.
\end{proposition}
\begin{proof}
Fix $0\ne u\in U_1$ and choose a non-zero nilpotent $x_u\in\mathfrak{sl}_2$ with
\[
\ker\rho_1(x_u)=\mathbb Cu.
\]
The pair $(x_u,u\otimes w)$ shows that
\[
\Delta(u\otimes w)\in u\otimes W_d.
\]
Thus
\[
\Delta(u\otimes w)=u\otimes F_u(w)
\]
for a homogeneous map $F_u:W_d\to W_d$.  Using one representing derivation for the pair $u\otimes w_1,u\otimes w_2$, the components transverse to the line $\mathbb Cu$ force the $\mathfrak{sl}_2$-part to preserve that line.  Hence the same operator $B+\lambda I$ on $W_d$ represents $F_u$ at $w_1$ and $w_2$.  Lemma~\ref{lem:2local-conformal-orthogonal} gives
\[
F_u=B_u+c_uI,
\qquad B_u\in\mathfrak{so}_d.
\]

Let $u,v$ be linearly independent.  Apply the two-local condition to $u\otimes w$ and $v\otimes w$.  Since both values preserve their first-factor lines, the representing $\mathfrak{sl}_2$-operator is diagonal with respect to the basis $u,v$.  Therefore
\[
(F_u-F_v)(w)\in\mathbb Cw
\qquad(w\in W_d).
\]
A linear operator for which every vector is an eigenvector is scalar.  Comparing skew and scalar parts gives
\[
B_u=B_v.
\]
Homogeneity handles proportional $u,v$, so all $B_u$ are one matrix $B_0$.  This proves \eqref{eq:m1-pure-tensor-form-higher}.

For $d\ge3$, use the pair $(R_{ij},u\otimes\varepsilon_i)$.  Since $\Delta(R_{ij})=0$, the rotation part of a representing derivation centralizes $R_{ij}$.  Comparing the components orthogonal to $\varepsilon_i$ in its value at $u\otimes\varepsilon_i$ (the first-factor operator contributes only in the $\varepsilon_i$-coordinate) shows from \eqref{eq:m1-pure-tensor-form-higher} that
\[
B_0\varepsilon_i\in\operatorname{span}\{\varepsilon_i,\varepsilon_j\}.
\]
Repeating with $R_{ik}$ for two distinct $j,k\ne i$ gives $B_0\varepsilon_i\in\mathbb C\varepsilon_i$, and skew-symmetry yields $B_0\varepsilon_i=0$.  Thus $B_0=0$.  If $d=2$, then $B_0\in\mathfrak{so}_2=Z(S)$ and $\operatorname{ad}B_0$ vanishes on $S$.
\end{proof}

\begin{proposition}\label{prop:higher-m1-radical-zero}
Let $d\ge2$ and $m=1$.  After subtracting a global derivation, every 2-local derivation may be normalized so that
\[
\Delta|_S=0,
\qquad
\Delta|_V=0,
\qquad
\Delta(\mathfrak g_{1/2}(d))\subseteq V.
\]
\end{proposition}
\begin{proof}
Start with Proposition~\ref{prop:higher-S-normalization} and then apply Proposition~\ref{prop:higher-m1-pure-tensors}.  Remove the common rotation term; for $d\ge3$ it is already zero, while for $d=2$ it is removed by $\operatorname{ad}B_0$ without changing $\Delta|_S$.

Thus for pure tensors
\[
\Delta(u\otimes w)=c([u])u\otimes w.
\]
We show that $c$ is constant.  Let $u,v$ be linearly independent and choose a nilpotent $x\in\mathfrak{sl}_2$, scaled so that
\[
\rho_1(x)u=0,
\qquad
\rho_1(x)v=u.
\]
Fix an index $i$ and put
\[
y=x+(u+v)\otimes\varepsilon_i.
\]
The pair $(x,y)$ shows that $\Delta(y)\in V$.  Use next the pair $(u\otimes\varepsilon_i,y)$.  Since the value at the first point is $c([u])u\otimes\varepsilon_i$, the rank-one tensor relation implies that the representing rotation sends $\varepsilon_i$ to a scalar multiple of $\varepsilon_i$; skew-symmetry at the non-isotropic vector $\varepsilon_i$ makes this scalar zero.  Thus the rotation part kills $\varepsilon_i$.  The vanishing of the $S$-component at $y$ forces the $\mathfrak{sl}_2$-part to commute with $x$, hence it is a scalar multiple of $x$.  As $xu=0$, the scalar outer coefficient is $c([u])$.  The coefficient of $v\otimes\varepsilon_i$ in $\Delta(y)$ is therefore $c([u])$.

Using instead the pair $(v\otimes\varepsilon_i,y)$, the equality at $v\otimes\varepsilon_i$ forces the multiple of $x$ to vanish because $xv=u$ is not proportional to $v$.  Hence the scalar outer coefficient is $c([v])$, and the coefficient of $v\otimes\varepsilon_i$ in the same vector $\Delta(y)$ is $c([v])$.  Therefore $c([u])=c([v])$.  Homogeneity covers proportional vectors, so $c$ is one scalar.  Subtract $c\delta_0$.

We now have $\Delta(u\otimes w)=0$ for every pure tensor.  Let
\[
v=\sum_{i=1}^d v_i\otimes\varepsilon_i\in V.
\]
Fix $i$.  If $v_i\ne0$, use the zero pure tensor $p_i=v_i\otimes\varepsilon_i$ as one point of the two-local pair $(p_i,v)$.  On $V$ the representing derivation has the form
\[
T\otimes I+I\otimes B.
\]
The equality at $p_i$ and the rank-one tensor argument give
\[
Tv_i=0,
\qquad
B\varepsilon_i=0.
\]
Skew-symmetry also kills the $i$th row of $B$, and hence the $i$th spatial coordinate of $\Delta(v)$ is zero.  If $v_i=0$, use any non-zero $u\otimes\varepsilon_i$ in place of $p_i$; the same conclusion follows.  Since $i$ is arbitrary, $\Delta(v)=0$ for all $v\in V$.
\end{proof}

\subsection{Rank-two mixed anchors in the standard module}
The obstruction in the standard module is that a single pure tensor does not separate the $\mathfrak{sl}_2$ and scalar parts.  A rank-two tensor does.  This produces mixed zero anchors that settle the entire case $d=2$ and provide the semisimple anchor family used later in the higher-dimensional completion.

For distinct indices $i,j$ put
\[
p_{ij}=u_0\otimes\varepsilon_i+u_1\otimes\varepsilon_j.
\]

\begin{lemma}\label{lem:m1-rank-two-zero-anchors}
Let $d\ge2$, $m=1$, and let $\Delta$ be a 2-local derivation satisfying
\[
\Delta|_S=0,\qquad \Delta|_V=0,\qquad \Delta(\mathfrak g_{1/2}(d))\subseteq V.
\]
Then, for every $i\ne j$,
\[
\Delta(e+p_{ij})=0,
\qquad
\Delta(f+p_{ij})=0.
\]
More generally, the same conclusion holds after simultaneously replacing $(\varepsilon_i,\varepsilon_j)$ by any orthonormal pair in $W_d$ and conjugating the $U_1$-factor by an inner automorphism of $\mathfrak{sl}_2$.
\end{lemma}
\begin{proof}
Set $z=e+p_{ij}$.  First use the pair $(p_{ij},z)$.  Since $\Delta(p_{ij})=0$, a representing derivation $D$ satisfies $D(p_{ij})=0$, and the $S$-component of $D(z)=\Delta(z)$ gives $[a,e]=0$ for the reductive part $a$ of $D$.  Hence
\[
\Delta(z)=D(e)=-\rho(e)r\in u_0\otimes W_d.
\]
Thus write
\[
\Delta(z)=u_0\otimes q.
\]

Next use the pair $(h,z)$.  The equality $D(h)=0$ forces the $\mathfrak{sl}_2$-part of the reductive coefficient to lie in $\mathbb Ch$ and, because $\rho_1(h)$ is invertible, forces the radical coefficient to be zero.  Since $D(z)$ has no $S$-component, the $h$-coefficient is also zero.  Consequently
\[
\Delta(z)
=u_0\otimes(B\varepsilon_i+\lambda\varepsilon_i)
+u_1\otimes(B\varepsilon_j+\lambda\varepsilon_j)
\]
for some $B\in\mathfrak{so}_d$ and $\lambda\in\mathbb C$.  The $u_1$-coefficient must vanish.  Taking the standard symmetric product with $\varepsilon_j$ gives $\lambda=0$, hence
\[
B\varepsilon_j=0,
\qquad
q=B\varepsilon_i.
\]
Therefore
\[
q\perp\varepsilon_i,\varepsilon_j.
\]
If $d=2$, this already gives $q=0$.

Assume $d\ge3$ and fix $k\notin\{i,j\}$.  Apply the two-local condition to the pair $(R_{jk},z)$, where $R_{jk}$ is the elementary rotation in the $(j,k)$-plane.  From $D(R_{jk})=0$ we obtain
\[
[B,R_{jk}]=0,
\qquad
R_{jk}r=0.
\]
As above, the vanishing of the $u_1$-coefficient of $D(z)$ gives $\lambda=0$ and $B\varepsilon_j=0$.  Commuting with $R_{jk}$ then yields $B\varepsilon_k=0$.  The $\varepsilon_k$-coordinate of the radical term $\rho(e)r$ is zero because $r$ is supported in $\ker R_{jk}$, while skew-symmetry gives
\[
\langle B\varepsilon_i,\varepsilon_k\rangle
=-\langle\varepsilon_i,B\varepsilon_k\rangle=0.
\]
Hence the $\varepsilon_k$-coordinate of $q$ is zero.  Since $k$ is arbitrary and the $i$- and $j$-coordinates were already zero, $q=0$.  Thus $\Delta(e+p_{ij})=0$.

The proof for $f+p_{ij}$ is symmetric: the pair with $p_{ij}$ first puts the value in $u_1\otimes W_d$, the pair with $h$ gives $B\varepsilon_i=0$, and the rotations $R_{ik}$ eliminate the remaining coordinates.

Finally, orthogonal changes of the spatial basis and inner automorphisms of $\mathfrak{sl}_2$ extend to automorphisms of $\mathfrak g_{1/2}(d)$.  Conjugating the normalized 2-local map preserves the normalization and therefore carries the zero anchors above to the asserted family.
\end{proof}

\begin{theorem}\label{thm:m1-d2-full}
Every 2-local derivation on
\[
\mathfrak g_{1/2}(2)
=(\mathfrak{sl}_2\oplus\mathfrak{so}_2)\ltimes(U_1\otimes\mathbb C^2)
\]
is a derivation.  Equivalently,
\[
\TwoLocDer\mathfrak g_{1/2}(2)=\Der\mathfrak g_{1/2}(2).
\]
\end{theorem}
\begin{proof}
By Proposition~\ref{prop:higher-m1-radical-zero}, after subtracting one global derivation we may assume
\begin{equation}\label{eq:m1-d2-normalized}
\Delta|_S=0,
\qquad
\Delta|_V=0,
\qquad
\Delta(\mathfrak g_{1/2}(2))\subseteq V.
\end{equation}
We prove that this normalized map is zero.

Let
\[
y=X+R+v,
\qquad
X\in\mathfrak{sl}_2,
\quad R\in\mathfrak{so}_2,
\quad v\in V,
\]
and put $Q=\Delta(y)$.  Inner conjugation on the $\mathfrak{sl}_2$-factor preserves the normalization, so a non-zero $X$ may be reduced either to a non-zero multiple of $h$ or to $e$.

Suppose first that $X=ch$ with $c\ne0$.  Use the zero anchors $e+p_{12}$ and $f+p_{12}$.  Since $\mathfrak{so}_2$ is one-dimensional, a skew-symmetric matrix that kills either $\varepsilon_1$ or $\varepsilon_2$ is zero.  For a derivation annihilating $e+p_{12}$, the condition that its $S$-component vanish at $y$ forces the $e$-coefficient to be zero because $[e,h]\ne0$.  Hence
\[
Q\in\rho(X+R)(u_0\otimes W_2).
\]
Similarly, the anchor $f+p_{12}$ gives
\[
Q\in\rho(X+R)(u_1\otimes W_2).
\]
Because $h$ preserves the two weight lines and $R$ acts only on the spatial factor, the first space lies in $u_0\otimes W_2$ and the second in $u_1\otimes W_2$.  Their intersection is zero, so $Q=0$.

Now let $X=e$.  Assume first that $R\ne0$.  For an inner automorphism $G$ of $\mathfrak{sl}_2$, Lemma~\ref{lem:m1-rank-two-zero-anchors} gives a zero anchor whose nilpotent $S$-part is
\[
s=GeG^{-1}
\]
and whose radical part has first factor $G$.  Whenever $[s,e]\ne0$, a derivation annihilating this anchor and representing $Q$ at $y$ has zero $\mathfrak{sl}_2$ coefficient and zero scalar coefficient; since $d=2$, its rotation coefficient also vanishes.  Therefore
\begin{equation}\label{eq:m1-d2-nilpotent-image}
Q\in\rho(e+R)(\ker\rho_1(s)\otimes W_2).
\end{equation}
As $s$ ranges over non-zero nilpotent elements not commuting with $e$, the lines $\ker\rho_1(s)$ range over all but at most one point of $\mathbb P(U_1)$.

Diagonalize the non-zero matrix $R$ on $W_2$.  Its two eigenvalues are $\nu$ and $-\nu$ with $\nu\ne0$.  On the corresponding eigenspaces, $\rho(e+R)$ acts on the first factor as
\[
\rho_1(e)+\nu I
\qquad\text{and}\qquad
\rho_1(e)-\nu I,
\]
respectively.  Both operators are invertible.  Hence, as the line $L=\ker\rho_1(s)$ varies, the lines $(\rho_1(e)\pm\nu I)L$ also vary through infinitely many projective lines, whose total intersection is zero.  Equation~\eqref{eq:m1-d2-nilpotent-image} therefore yields $Q=0$.

It remains in the nilpotent case to take $R=0$.  Write
\[
v=u_0\otimes v_0+u_1\otimes v_1.
\]
Let $w_0,w_1$ be any orthonormal basis of $W_2$ and set
\[
p(w_0,w_1)=u_0\otimes w_0+u_1\otimes w_1.
\]
By Lemma~\ref{lem:m1-rank-two-zero-anchors},
\[
\Delta(e+p(w_0,w_1))=0.
\]
A derivation annihilating this anchor has, in dimension two, no rotation or scalar coefficient.  Its $\mathfrak{sl}_2$-part is $\alpha e$, and the $u_1$-row of its radical coefficient is $\alpha w_1$.  Evaluating the same derivation at $e+v$ gives
\[
Q\in u_0\otimes\mathbb C(v_1-w_1).
\]
This holds for every unit vector $w_1$ occurring in an orthonormal basis.  The affine conic
\[
\{w\in W_2:\langle w,w\rangle=1\}
\]
is not contained in an affine line.  Hence we can choose two such vectors $w_1,w_1'$ for which $v_1-w_1$ and $v_1-w_1'$ are linearly independent.  The two one-dimensional restrictions on $Q$ then have zero intersection, so $Q=0$.

Finally suppose $X=0$.  If $R=0$, then $y=v\in V$ and $Q=0$ by normalization.  If $R\ne0$, the two elements
\[
e+R+v,
\qquad
f+R+v
\]
have already been treated and therefore have zero $\Delta$-value.  Apply the 2-local condition first to $(R+v,e+R+v)$.  The difference of the two values is $D(e)=-Q$, so the fact that $Q\in V$ implies
\[
Q\in\operatorname{Im}\rho_1(e)\otimes W_2=u_0\otimes W_2.
\]
Using $(R+v,f+R+v)$ similarly gives
\[
Q\in u_1\otimes W_2.
\]
Thus $Q=0$.

The normalized map vanishes on every element, and undoing the initial subtraction proves the theorem.
\end{proof}

\begin{corollary}\label{cor:noncentral-dle2-all}
For every positive half-integer $\ell$,
\[
\TwoLocDer\mathfrak g_\ell(d)=\Der\mathfrak g_\ell(d)
\qquad(d=1,2).
\]
Together with Corollary~\ref{cor:noncentral-all-mge2}, this leaves only the standard-module case $2\ell=1$ in spatial dimension $d\ge3$.
\end{corollary}

\begin{lemma}\label{lem:m1-higher-nilpotent-row}
Assume $d\ge2$ and let $\Delta$ satisfy the normalization of Proposition~\ref{prop:higher-m1-radical-zero}.  Write
\[
v=u_0\otimes v_0+u_1\otimes v_1.
\]
Then
\[
\Delta(e+v)\in u_0\otimes W_d,
\qquad
\Delta(f+v)\in u_1\otimes W_d.
\]
\end{lemma}
\begin{proof}
Write
\[
\Delta(e+v)=u_0\otimes q_0+u_1\otimes q_1.
\]
Choose any orthonormal pair $w_0,w_1$ in $W_d$ and use the zero anchor
\[
e+u_0\otimes w_0+u_1\otimes w_1.
\]
For a representing derivation annihilating this anchor and evaluated at $e+v$, the $\mathfrak{sl}_2$-part is $\alpha e$, the scalar coefficient is zero, the rotation coefficient $B$ satisfies $Bw_1=0$, and the $u_1$-row of the radical coefficient is $Bw_0+\alpha w_1$.  The $u_1$-row of the value at $e+v$ is therefore simply
\[
q_1=Bv_1.
\]
Skew-symmetry and $Bw_1=0$ imply
\[
\langle q_1,w_1\rangle=0.
\]
The vector $q_1$ is independent of the chosen anchor, while $w_1$ may be chosen arbitrarily on the quadric $\{w\in W_d:\langle w,w\rangle=1\}$.  This quadric is not contained in any proper linear hyperplane of $W_d$.  Hence $\langle q_1,w_1\rangle=0$ for every such $w_1$ implies $q_1=0$.  The assertion for $f+v$ is symmetric.
\end{proof}

\begin{lemma}\label{lem:m1-h-rank-two-zero-anchors}
Let $d\ge2$, $m=1$, and let $\Delta$ satisfy the normalization of Proposition~\ref{prop:higher-m1-radical-zero}.  For every orthonormal pair $w_0,w_1\in W_d$, put
\[
p(w_0,w_1)=u_0\otimes w_0+u_1\otimes w_1.
\]
Then
\[
\Delta\bigl(h+p(w_0,w_1)\bigr)=0.
\]
\end{lemma}
\begin{proof}
It is enough to work first with $w_0=\varepsilon_i$ and $w_1=\varepsilon_j$, $i\ne j$, because orthogonal changes of the spatial basis preserve the normalization.  Put
\[
z=h+p_{ij},\qquad Q=\Delta(z)=u_0\otimes q_0+u_1\otimes q_1.
\]
Use the zero anchor $e+p_{ij}$ from Lemma~\ref{lem:m1-rank-two-zero-anchors}.  Let
\[
D=\operatorname{ad}(a+r)+\lambda\delta_0
\]
represent the pair $(e+p_{ij},z)$.  From $D(e+p_{ij})=0$ one obtains, exactly as in the proof of Lemma~\ref{lem:m1-rank-two-zero-anchors},
\[
a_{\mathfrak{sl}_2}=\alpha e,\qquad \lambda=0,\qquad B\varepsilon_j=0,
\]
and, writing $r=u_0\otimes r_0+u_1\otimes r_1$,
\[
r_1=B\varepsilon_i+\alpha\varepsilon_j.
\]
Since $Q$ has no reductive component, the reductive component of $D(z)$ is
\[
[\alpha e,h]=-2\alpha e,
\]
so $\alpha=0$.  Therefore
\[
q_1=B\varepsilon_i.
\]
Skew-symmetry and $B\varepsilon_j=0$ give
\[
q_1\perp \varepsilon_i,\varepsilon_j.
\]
Using instead the zero anchor $f+p_{ij}$ gives symmetrically
\[
q_0\perp \varepsilon_i,\varepsilon_j.
\]
Thus $Q=0$ when $d=2$.

Assume $d\ge3$ and fix $k\notin\{i,j\}$.  Apply the two-local condition to $(R_{ik},z)$.  If $D$ is a representing derivation, then $D(R_{ik})=0$ implies
\[
[B,R_{ik}]=0,\qquad R_{ik}r=0.
\]
Hence both rows of $r$ have zero $\varepsilon_i$- and $\varepsilon_k$-coordinates.  Moreover, a skew-symmetric matrix commuting with $R_{ik}$ preserves the decomposition
\[
\operatorname{span}\{\varepsilon_i,\varepsilon_k\}\oplus
\ker R_{ik}.
\]
Since $\varepsilon_j\in\ker R_{ik}$, the $\varepsilon_k$-coordinate of $B\varepsilon_j$ is zero.  Inspecting the $u_1\otimes\varepsilon_k$-coefficient of $D(z)=Q$ therefore yields
\[
\langle q_1,\varepsilon_k\rangle=0.
\]
Similarly, using the pair $(R_{jk},z)$ gives
\[
\langle q_0,\varepsilon_k\rangle=0.
\]
Together with the already known vanishing of the $i$- and $j$-coordinates, this gives $q_0=q_1=0$.  Hence $\Delta(z)=0$.
\end{proof}

\begin{theorem}\label{thm:m1-zero-rotation-full}
Let $d\ge3$ and let $\Delta$ satisfy the normalization of Proposition~\ref{prop:higher-m1-radical-zero}.  Then
\[
\Delta(X+v)=0
\qquad
(X\in\mathfrak{sl}_2,\ v\in U_1\otimes W_d).
\]
In other words, every normalized 2-local derivation vanishes on all elements whose $\mathfrak{so}_d$-component is zero.
\end{theorem}
\begin{proof}
The cases $X=0$ and $v=0$ follow from the normalization.  By homogeneity and inner conjugation on the $\mathfrak{sl}_2$-factor, a non-zero $X$ is reduced either to $h$ or to $e$.

First let $y=h+v$, where
\[
v=u_0\otimes v_0+u_1\otimes v_1,
\qquad
\Delta(y)=u_0\otimes q_0+u_1\otimes q_1.
\]
Fix an arbitrary orthonormal pair $w_0,w_1$.  Use the zero anchor
\[
e+p(w_0,w_1).
\]
The same coefficient computation as in Lemma~\ref{lem:m1-h-rank-two-zero-anchors}, now evaluated at $h+v$, gives
\[
q_1=B(v_1+w_0),\qquad Bw_1=0
\]
for some $B\in\mathfrak{so}_d$.  Hence
\[
\langle q_1,w_1\rangle=0.
\]
The vector $q_1$ is independent of the anchor, while $w_1$ is arbitrary; therefore $q_1=0$.  Using the zero anchor $f+p(w_0,w_1)$ similarly gives
\[
q_0=B'(v_0-w_1),\qquad B'w_0=0,
\]
and hence $\langle q_0,w_0\rangle=0$ for every unit $w_0$.  Thus $q_0=0$ and
\[
\Delta(h+v)=0.
\]

Now let $y=e+v$.  By Lemma~\ref{lem:m1-higher-nilpotent-row},
\[
\Delta(y)=u_0\otimes q
\]
for some $q\in W_d$.  By the first part of the proof, every element $h+w$ with $w\in V$ is a zero point.  Choose
\[
w=u_0\otimes w_0+u_1\otimes w_1
\]
with $w_1=-v_0$ and arbitrary $w_0$, and use the pair $(h+w,e+v)$.  Since the $\mathfrak{sl}_2$-parts $h$ and $e$ have zero common centralizer, the $\mathfrak{sl}_2$-coefficient of the representing derivation is zero.  Writing its remaining action on $V$ as $I\otimes B+\lambda I$, put $A=B+\lambda I$.  The equality $D(h+w)=0$ gives
\[
r=h^{-1}Aw=hAw,
\]
that is,
\[
r_0=Aw_0,\qquad r_1=-Aw_1.
\]
Consequently
\[
D(e+v)
=u_0\otimes A(v_0+w_1)+u_1\otimes Av_1.
\]
Its $u_1$-row must vanish by Lemma~\ref{lem:m1-higher-nilpotent-row}; with $w_1=-v_0$ the $u_0$-row vanishes as well.  Hence $\Delta(e+v)=0$.  The case of a non-zero nilpotent element conjugate to $e$ follows by inner conjugation.  This completes the proof.
\end{proof}

\begin{proposition}\label{prop:m1-universal-image-intersection}
Let $d\ge3$ and let $\Delta$ satisfy the normalization of Proposition~\ref{prop:higher-m1-radical-zero}.  For
\[
y=X+R+v,
\qquad X\in\mathfrak{sl}_2,\ R\in\mathfrak{so}_d,\ v\in V,
\]
put $Q=\Delta(y)$.  Then
\begin{equation}\label{eq:m1-universal-image-intersection}
Q\in
\bigcap_{T\in\mathfrak{sl}_2}
\operatorname{Im}\bigl(\rho_1(T)\otimes I-I\otimes R\bigr).
\end{equation}
\end{proposition}
\begin{proof}
Fix $s\in\mathfrak{sl}_2$.  By Theorem~\ref{thm:m1-zero-rotation-full},
\[
\Delta(s+v)=0.
\]
Let $D=\operatorname{ad}(a+r)+\lambda\delta_0$ represent the pair $(s+v,y)$.  From $D(s+v)=0$ we have
\[
(\rho(a)+\lambda I)v=\rho(s)r.
\]
The reductive component of $D(y)=Q\in V$ vanishes, so the reductive coefficients commute with $X$ and $R$.  Therefore
\[
Q=(\rho(a)+\lambda I)v-\rho(X+R)r
=\rho(s-X-R)r.
\]
As $s$ ranges through $\mathfrak{sl}_2$, so does $T=s-X$, and \eqref{eq:m1-universal-image-intersection} follows.
\end{proof}

\begin{theorem}\label{thm:m1-semisimple-rotation}
Let $d\ge3$ and let $R\in\mathfrak{so}_d$ be semisimple as an endomorphism of $W_d$.  Under the normalization of Proposition~\ref{prop:higher-m1-radical-zero},
\[
\Delta(X+R+v)=0
\]
for every $X\in\mathfrak{sl}_2$ and $v\in V$.
\end{theorem}
\begin{proof}
Let $Q$ be the value at $X+R+v$.  By Proposition~\ref{prop:m1-universal-image-intersection}, it is enough to intersect the images corresponding to $T=t h$.  Since $R$ is semisimple,
\[
W_d=\bigoplus_{\mu}W_\mu,
\qquad
R|_{W_\mu}=\mu I.
\]
Write
\[
Q=u_0\otimes q_0+u_1\otimes q_1,
\qquad
q_i=\sum_\mu q_{i,\mu},\quad q_{i,\mu}\in W_\mu.
\]
On $u_0\otimes W_d$ the operator $\rho_1(th)\otimes I-I\otimes R$ is $tI-R$, whereas on $u_1\otimes W_d$ it is $-tI-R$.  Taking $t=\mu$, the image of $tI-R$ has zero $W_\mu$-component, so $q_{0,\mu}=0$.  Taking $t=-\mu$ gives $q_{1,\mu}=0$.  This also covers $\mu=0$.  Hence $Q=0$.
\end{proof}

\begin{lemma}\label{lem:so-semisimple-perturbation}
Let $d\ge3$, let $W_d=\mathbb C^d$ carry its standard non-degenerate symmetric bilinear form $\langle\cdot,\cdot\rangle$, and let $R\in\mathfrak{so}_d$.  Then
\begin{equation}\label{eq:so-semisimple-perturbation}
\bigcap_{\substack{A\in\mathfrak{so}_d\\ A\text{ semisimple}}}
\operatorname{Im}(A-R)=\{0\}.
\end{equation}
\end{lemma}
\begin{proof}
For a skew-symmetric endomorphism $B$ one has
\[
\operatorname{Im}B=(\ker B)^\perp.
\]
Fix $0\ne q\in W_d$.  It is enough to construct a semisimple $A\in\mathfrak{so}_d$ and a vector $w\in\ker(A-R)$ with $\langle q,w\rangle\ne0$.

Choose a non-isotropic vector $w$ such that $\langle q,w\rangle\ne0$; such vectors form a non-empty Zariski open subset of $W_d$.  Put
\[
u=Rw.
\]
Then $\langle u,w\rangle=0$.  If $u=0$, take $A=0$.  If $\langle u,u\rangle\ne0$, put $a=\langle w,w\rangle$ and define $A$ on $\operatorname{span}\{w,u\}$ by
\[
Aw=u,
\qquad
Au=-\frac{\langle u,u\rangle}{a}\,w,
\]
and let $A$ vanish on the orthogonal complement.  This $A$ is skew-symmetric.  On $\operatorname{span}\{w,u\}$ one has
\[
A^2=-\frac{\langle u,u\rangle}{a}I,
\]
with non-zero scalar on the right, so the restriction is diagonalizable over $\mathbb C$; together with the zero action on the orthogonal complement, $A$ is semisimple.

It remains to consider $0\ne u$ isotropic.  Since $w^\perp$ is non-degenerate and has dimension at least two, first choose $v_0\in w^\perp$ with $\langle u,v_0\rangle=1$ and then replace it by $v=v_0-\tfrac12\langle v_0,v_0\rangle u$.  Thus $v$ is isotropic and $\langle u,v\rangle=1$.  For $x,y\in W_d$ write
\[
(x\wedge y)(z)=\langle y,z\rangle x-\langle x,z\rangle y.
\]
For any $t\ne0$ set
\[
A=\frac1a(u\wedge w)+t(u\wedge v)
\]
on $E=\operatorname{span}\{w,u,v\}$ and let it vanish on $E^\perp$.  Then $A\in\mathfrak{so}_d$ and $Aw=u=Rw$.  In the ordered basis $(u,w,v)$ its restriction to $E$ is upper triangular with diagonal entries $t,0,-t$; these are distinct because $t\ne0$.  Hence $A|_E$ is diagonalizable, and the direct sum with the zero action on the non-degenerate complement $E^\perp$ is semisimple.  In every case $w\in\ker(A-R)$ and $\langle q,w\rangle\ne0$, so $q\notin\operatorname{Im}(A-R)$.  Since $q$ was arbitrary, \eqref{eq:so-semisimple-perturbation} follows.
\end{proof}

\begin{theorem}\label{thm:m1-higher-full}
Let $d\ge3$.  Every 2-local derivation on $\mathfrak g_{1/2}(d)$ is a derivation.
\end{theorem}
\begin{proof}
Normalize $\Delta$ as in Proposition~\ref{prop:higher-m1-radical-zero}, so that $\Delta$ vanishes on $S=\mathfrak{sl}_2\oplus\mathfrak{so}_d$ and on $V=U_1\otimes W_d$.  Let
\[
y=X+R+v,
\qquad X\in\mathfrak{sl}_2,
\quad R\in\mathfrak{so}_d,
\quad v\in V,
\]
and put $Q=\Delta(y)$.  The preceding reductions imply $Q\in V$.

For every semisimple $A\in\mathfrak{so}_d$, Theorem~\ref{thm:m1-semisimple-rotation} gives
\[
\Delta(X+A+v)=0.
\]
Let $D$ be a derivation representing the pair $(X+A+v,y)$.  By linearity of $D$,
\[
Q=D(y)-D(X+A+v)=D(R-A).
\]
The scalar outer derivation vanishes on $\mathfrak{so}_d$.  Since $Q$ has no reductive component, only the radical part of the inner element can contribute to the $V$-component of $D(R-A)$; consequently
\[
Q\in U_1\otimes\operatorname{Im}(A-R).
\]
This holds for every semisimple $A\in\mathfrak{so}_d$.  Lemma~\ref{lem:so-semisimple-perturbation} therefore yields $Q=0$.  Thus the normalized map is identically zero, and the original 2-local derivation is the global derivation that was subtracted in the normalization.
\end{proof}

\begin{corollary}\label{cor:noncentral-all}
For every $d\ge1$ and every positive half-integer $\ell$,
\begin{equation}\label{eq:noncentral-all}
\TwoLocDer\mathfrak g_\ell(d)=\Der\mathfrak g_\ell(d).
\end{equation}
\end{corollary}
\begin{proof}
The case $d=1$ is Corollary~\ref{cor:one-dimensional-all}.  For $d\ge2$ and $2\ell\ge3$ use Theorem~\ref{thm:higher-full-mge3}; for $2\ell=2$ use Theorem~\ref{thm:higher-m2-full}; for $2\ell=1$, use Theorem~\ref{thm:m1-d2-full} when $d=2$ and Theorem~\ref{thm:m1-higher-full} when $d\ge3$.
\end{proof}

\section{Common reductions for the central extensions}\label{sec:centralreductions}
We now develop the normalization that is common to the mass and exotic central extensions.  The key advantage is that the scalar outer derivation acts nontrivially on the center, so a single central value fixes its coefficient in every pair containing the center.

\subsection{Normalization at the center}
We first show that the center determines the scalar outer component of a 2-local derivation globally.

\begin{proposition}\label{prop:center-normalization}
Let $L$ be either $\widehat{\g}^{\,M}_\ell(d)$ or $\widehat{\g}^{\,\Theta}_\ell(2)$, and let $z$ denote the corresponding central generator. If $\Delta$ is a 2-local derivation on $L$, then
\[
\Delta(z)=\kappa z
\]
for some $\kappa\in\C$.  After subtracting $\frac{\kappa}{2}\delta_M$ in the mass case, or $\frac{\kappa}{2}\delta_\Theta$ in the exotic case, we may assume
\[
\Delta(z)=0.
\]
Under this normalization, for every $x\in L$ there is a representing derivation for the pair $(z,x)$ whose scalar outer coefficient is zero.
\end{proposition}
\begin{proof}
In both central extensions, inner derivations annihilate the center, the central shear (when present) also annihilates the center, and the scalar outer derivation sends $z$ to $2z$. Thus for a pair $(z,x)$, writing the representing derivation in the form \eqref{eq:der-mass} or \eqref{eq:der-exotic}, we obtain
\[
\Delta(z)=2\lambda_{z,x}z.
\]
The left-hand side is independent of $x$, so $\lambda_{z,x}$ is fixed.  Subtracting the corresponding scalar derivation gives $\Delta(z)=0$.  For a representing derivation of a pair $(z,x)$ after this normalization, the equality $D(z)=0$ forces its scalar outer coefficient to vanish.
\end{proof}

\subsection{Reduction to the radical}
We next combine the central normalization with the quotient lemma to isolate the genuinely new part of the higher-dimensional problem.

\begin{proposition}\label{prop:central-radical-reduction}
Let $L$ be a mass or exotic central extension and let $\mathfrak h$ be its Heisenberg nilradical.  After subtracting a global derivation, any 2-local derivation $\Delta$ may be assumed to satisfy
\[
\Delta(z)=0,
\qquad
\Delta(L)\subseteq\mathfrak h.
\]
Thus the entire reductive quotient map is zero, including the central $\mathfrak{so}_2$-direction when $d=2$.
\end{proposition}
\begin{proof}
The nilradical $\mathfrak h$ is characteristic, so Lemma~\ref{lem:quotient} applies.  Proposition~\ref{prop:center-normalization} first removes the scalar outer component detected by the center.  For $d=1$ and $d\ge3$, the induced map on the semisimple quotient is a derivation by \cite{AKR}, and subtracting its inner lift kills the quotient map.  In spatial dimension two, Lemma~\ref{lem:d2-reductive-quotient} shows directly that the full quotient map on $\mathfrak{sl}_2\oplus\mathbb Cs$ is one inner derivation on $\mathfrak{sl}_2$ and zero on $s$; subtracting its lift again kills the entire quotient map.  The central shear takes values in $\mathfrak h$ and hence does not affect this conclusion.
\end{proof}

\subsection{Quotient-to-center reduction}
The preceding quotient normalization can be sharpened whenever the corresponding non-central quotient has already been proved rigid.  The point is that center normalization also rules out a residual scalar derivation on the quotient.

\begin{proposition}\label{prop:central-center-valued-reduction}
Let $L$ be one of the following:
\begin{enumerate}[label=\textup{(\roman*)}]
\item the mass extension $\widehat{\mathfrak g}^{\,M}_\ell(d)$ with $\ell\in\N-\tfrac12$;
\item the exotic extension $\widehat{\mathfrak g}^{\,\Theta}_\ell(2)$ with $\ell\in\mathbb N$.
\end{enumerate}
Then, after subtracting a global derivation, every 2-local derivation $\Delta$ may be normalized so that
\[
\Delta(z)=0,
\qquad
\Delta(L)\subseteq\mathbb Cz,
\]
where $z=M$ in the mass case and $z$ is the exotic central generator in the second case.
\end{proposition}
\begin{proof}
Start with the center normalization of Proposition~\ref{prop:center-normalization}.  The induced map on $L/\mathbb Cz$ is a 2-local derivation.  In the mass case, the quotient is $\mathfrak g_\ell(d)$ and is rigid for every admissible $\ell$ by Corollary~\ref{cor:noncentral-all}.  In the exotic case the quotient is $\mathfrak g_\ell(2)$ with $2\ell\ge2$, and rigidity follows from Corollary~\ref{cor:noncentral-dle2-all}.  Hence
\[
\overline\Delta=\operatorname{ad}(\overline b)+\alpha\delta_0
\]
for some $\overline b$ and $\alpha\in\mathbb C$.

Subtract an inner lift of $\operatorname{ad}(\overline b)$; the value on the center remains zero. We prove uniformly that $\alpha=0$. Put $m=2\ell$ and choose
\[
x=e+u_m\otimes\varepsilon_1.
\]
Let $D$ represent the pair $(z,x)$. Since $D(z)=0$, its scalar outer coefficient is zero. Thus $D=\ad(a_S+a_W)$, with an additional central shear allowed when $d=2$; this shear annihilates $x$. Modulo the center, $D(x)=\alpha u_m\otimes\varepsilon_1$. The $S$-component gives $[a_S,e]=0$, so the $\mathfrak{sl}_2$-part of $a_S$ is $\beta e$. If $B$ is its spatial rotation part, the radical component is
\[
-\rho(e)a_W+\beta m u_{m-1}\otimes\varepsilon_1+u_m\otimes B\varepsilon_1.
\]
The image of $\rho(e)$ has no $u_m$-row. Comparing that row therefore gives $B\varepsilon_1=\alpha\varepsilon_1$. Skew-symmetry and $\langle\varepsilon_1,\varepsilon_1\rangle=1$ yield $\alpha=0$. Hence the quotient map is zero, and all values of the residual map lie in $\C z$.
\end{proof}

\begin{proposition}\label{prop:mass-center-free-hyperplane}
Let $m=2\ell$ be any positive odd integer and let $\Delta$ be a 2-local derivation on the mass extension, normalized as in Proposition~\ref{prop:central-center-valued-reduction}.  If $d=2$, subtract in addition a scalar multiple of $\eta_M$ so that $\Delta(s)=0$.  Then
\[
\Delta(x)=0
\qquad
\text{for every }x\text{ whose }M\text{-coordinate is zero}.
\]
\end{proposition}
\begin{proof}
Put $m=2\ell$.  Since $\ell$ is half-integral in the mass case, $m$ is odd.  Suppose first that $d\ne2$.  Because $\Delta(h)\in\mathbb CM$ and no derivation can send $h$ to a non-zero central vector, we have $\Delta(h)=0$.  For an element $x$ with zero $M$-coordinate, represent the pair $(h,x)$ by
\[
D=\operatorname{ad}(a)+\lambda\delta_M.
\]
The equality $D(h)=0$ separates into reductive and radical components.  In particular, the non-central radical part $a_V$ satisfies
\[
[h,a_V]=0.
\]
The $h$-weights on $V$ are $m,m-2,\ldots,-m$, all of which are non-zero because $m$ is odd.  Hence $a_V=0$.  The remaining terms $\operatorname{ad}(a_S)$ and $\lambda\delta_M$ cannot create a central component from an element with zero $M$-coordinate.  Since $D(x)=\Delta(x)$ is central, it follows that $\Delta(x)=0$.

Now let $d=2$.  After subtracting $\Delta(s)\eta_M$, we have $\Delta(s)=0$.  Represent $(s,x)$ by
\[
D=\operatorname{ad}(a)+\lambda\delta_M+\mu\eta_M.
\]
The equality $D(s)=0$ gives
\[
[a_V,s]+\mu M=0.
\]
The two summands lie in the non-central radical and in the center, respectively.  Moreover, $s$ acts without zero eigenvalues on every spatial copy of the radical.  Hence $a_V=0$ and $\mu=0$.  Again the remaining terms have no central component on an element with zero $M$-coordinate, so $\Delta(x)=0$.
\end{proof}

\subsection{Symplectic completion of the mass-central case}
We use one flexible symplectic zero anchor to eliminate all central translates at once.  Put $m=2\ell$, $W=V_{\ell,d}$, and let $\omega_M$ be the non-degenerate alternating form determined by
\begin{equation}\label{eq:mass-symplectic-form}
[u,v]=\omega_M(u,v)M.
\end{equation}
For $t\in S:=\mathfrak{sl}_2\oplus\mathfrak{so}_d$, with $T=\rho(t)$, the Jacobi identity gives
\begin{equation}\label{eq:mass-symplectic-invariance}
\omega_M(Tu,v)+\omega_M(u,Tv)=0
\qquad(u,v\in W).
\end{equation}

\begin{lemma}\label{lem:mass-flexible-anchor}
Let $t_0\in S$ be regular semisimple and suppose that $T_0=\rho(t_0)$ is invertible on $W$.  In the case $d=2$, write $\tau$ for the $s$-coordinate of $t_0$; for $d\ne2$ put $\tau=0$.  Suppose
\[
D=\operatorname{ad}(a_S+a_W)+\lambda\delta_M
\]
for $d\ne2$, and
\[
D=\operatorname{ad}(a_S+a_W)+\lambda\delta_M+\mu\eta_M
\]
for $d=2$.  If $D(t_0+u)=0$, then, with $A=\rho(a_S)$,
\begin{equation}\label{eq:mass-flexible-anchor-relations}
[a_S,t_0]=0,
\qquad
a_W=T_0^{-1}(A+\lambda I)u,
\end{equation}
and
\begin{equation}\label{eq:mass-flexible-central-relation}
\lambda B_{T_0}(u,u)+\tau\mu=0,
\qquad
B_{T_0}(x,y)=\omega_M(T_0^{-1}x,y).
\end{equation}
For $d\ne2$ the term $\tau\mu$ is absent.
\end{lemma}
\begin{proof}
The $S$-component gives $[a_S,t_0]=0$, hence $A$ commutes with $T_0$.  The $W$-component gives the second relation in \eqref{eq:mass-flexible-anchor-relations}.  Since both $A$ and $T_0$ are skew-adjoint with respect to $\omega_M$ and commute, $T_0^{-1}A$ is self-adjoint with respect to $\omega_M$; therefore
\[
\omega_M(T_0^{-1}Au,u)=0.
\]
The central component of $D(t_0+u)=0$ is consequently
\[
0=\omega_M(a_W,u)+\tau\mu
 =\lambda\omega_M(T_0^{-1}u,u)+\tau\mu,
\]
which is \eqref{eq:mass-flexible-central-relation}.
\end{proof}

\begin{theorem}\label{thm:mass-mge3-full}
Let $\ell\in\mathbb N-\tfrac12$ with $2\ell\ge3$, and let $d\ge1$.  Every 2-local derivation on $\widehat{\mathfrak g}^{\,M}_\ell(d)$ is a derivation.
\end{theorem}
\begin{proof}
Put $m=2\ell$, so $m$ is odd, and normalize $\Delta$ as in Proposition~\ref{prop:mass-center-free-hyperplane}.  Thus $\Delta$ is center-valued and vanishes on the hyperplane with zero $M$-coordinate.  Let
\[
y=t+v+cM,
\qquad t\in S,
\quad v\in W,
\quad c\in\mathbb C,
\]
and write $\Delta(y)=qM$.  There is nothing to prove when $c=0$.

First assume $d\ne2$.  Choose a regular semisimple $t_0\in S$ such that $T_0=\rho(t_0)$ is invertible.  Such an element exists: for $d=1$ one may take a non-zero multiple of $h$, while for $d\ge3$ one may take $t_0=ah+R_0$ with $R_0\in\mathfrak{so}_d$ regular semisimple and choose $a$ outside the finite set for which one of the tensor-product eigenvalues vanishes.  Put
\[
r=\dim C_S(t_0)=1+\Big\lfloor\frac d2\Big\rfloor,
\qquad
n=\dim W=(m+1)d.
\]
Since $m\ge3$, one has $n>2r$.

For $a\in C_S(t_0)$ let $A=\rho(a)$ and define
\[
E=\operatorname{span}\{T_0^{-1}Av: a\in C_S(t_0)\}.
\]
Then $\dim E\le r$.  Hence the symplectic orthogonal complement $E^{\perp_{\omega_M}}$ has dimension at least $n-r>n/2$.  The form $B_{T_0}$ is non-degenerate and symmetric, so a totally isotropic subspace for $B_{T_0}$ has dimension at most $\lfloor n/2\rfloor$.  We may therefore choose
\[
u\in E^{\perp_{\omega_M}},
\qquad
B_{T_0}(u,u)\ne0.
\]
The anchor $t_0+u$ has zero $M$-coordinate, hence its value under $\Delta$ is zero.  Let $D$ represent the pair $(t_0+u,y)$.  Lemma~\ref{lem:mass-flexible-anchor} gives $\lambda=0$ and $a_W=T_0^{-1}Au$ for some $A$ coming from $C_S(t_0)$.  Since $D(y)=qM$, its central coefficient is
\[
q=\omega_M(T_0^{-1}Au,v).
\]
Because $T_0^{-1}A$ is $\omega_M$-self-adjoint,
\[
q=\omega_M(u,T_0^{-1}Av)=0
\]
by the choice of $u$.  Thus $q=0$.

Now let $d=2$ and write the reductive component of the target as
\[
t=t_1+\sigma s,
\qquad t_1\in\mathfrak{sl}_2.
\]
If $\sigma=0$, take $t_0=h$.  Its action on $W$ is invertible because $m$ is odd.  The same dimension argument, now with $\dim C_S(h)=2$ and $n=2(m+1)>4$, provides $u\in E^{\perp_{\omega_M}}$ with $B_{T_0}(u,u)\ne0$.  Lemma~\ref{lem:mass-flexible-anchor} gives $\lambda=0$; the shear coefficient may remain free, but $\eta_M(y)=0$ because $\sigma=0$.  Hence the same self-adjointness computation gives $q=0$.

Finally suppose $\sigma\ne0$.  Since the case $c=0$ is already known, assume $c\ne0$.  Choose $\tau\ne0$ outside the finite set for which
\[
t_0=h+\tau s
\]
fails to act invertibly on $W$.  Then $t_0$ is regular semisimple and $\dim C_S(t_0)=2$.  With $A=\rho(a)$ for $a\in C_S(t_0)$, enlarge the preceding auxiliary space to
\[
E'=\operatorname{span}\Bigl(\{T_0^{-1}Av:a\in C_S(t_0)\}\cup\{T_0^{-1}v\}\Bigr).
\]
Then $\dim E'\le3$.  Since $n=2(m+1)\ge8$, the symplectic orthogonal complement satisfies
\[
\dim (E')^{\perp_{\omega_M}}\ge n-3>\frac n2.
\]
The non-degenerate symmetric form $B_{T_0}$ cannot vanish identically on a subspace of dimension greater than $n/2$.  Hence we can choose a non-isotropic vector in $(E')^{\perp_{\omega_M}}$ and scale it, over $\C$, so that
\begin{equation}\label{eq:mass-d2-prescribed-B}
B_{T_0}(u,u)=\frac{2c\tau}{\sigma}.
\end{equation}
Let $D$ represent $(t_0+u,y)$.  Lemma~\ref{lem:mass-flexible-anchor} gives
\[
\mu=-\frac{\lambda B_{T_0}(u,u)}{\tau},
\qquad
a_W=T_0^{-1}(A+\lambda I)u
\]
for some $A=\rho(a)$ with $a\in C_S(t_0)$.  The two possible inner central contributions both vanish.  Indeed, self-adjointness of $T_0^{-1}A$ gives
\[
\omega_M(T_0^{-1}Au,v)
 =\omega_M(u,T_0^{-1}Av)=0,
\]
while $T_0^{-1}$ is $\omega_M$-skew-adjoint and therefore
\[
\omega_M(T_0^{-1}u,v)
 =-\omega_M(u,T_0^{-1}v)=0.
\]
Both equalities use $u\in(E')^{\perp_{\omega_M}}$.  Consequently
\[
q=2c\lambda+\sigma\mu
 =\lambda\left(2c-\frac{\sigma}{\tau}B_{T_0}(u,u)\right)=0
\]
by \eqref{eq:mass-d2-prescribed-B}.  Thus the normalized map is zero in all cases.
\end{proof}

\begin{theorem}\label{thm:mass-all}
For every $d\ge1$ and every $\ell\in\mathbb N-\tfrac12$,
\begin{equation}\label{eq:mass-all}
\TwoLocDer\widehat{\mathfrak g}^{\,M}_\ell(d)
=\Der\widehat{\mathfrak g}^{\,M}_\ell(d).
\end{equation}
\end{theorem}
\begin{proof}
For $2\ell\ge3$ this is Theorem~\ref{thm:mass-mge3-full}. We give a self-contained proof of the remaining endpoint $\ell=\tfrac12$, which is the Schr\"odinger algebra. This endpoint is also treated in \cite{TangWangJLT}; the argument below does not require that external result.

Put $m=1$, $W=U_1\otimes\C^d$, and $S=\mathfrak{sl}_2\oplus\mathfrak{so}_d$. By Propositions~\ref{prop:central-center-valued-reduction} and~\ref{prop:mass-center-free-hyperplane}, subtracting a global derivation gives a center-valued map $\Delta$ that vanishes on every element with zero $M$-coordinate. For $d=2$ we also have $\Delta(s)=0$. Fix
\[
y=t+v+cM,\qquad t\in S,\quad v\in W,\quad c\ne0,
\qquad \Delta(y)=qM.
\]
We show that $q=0$ by choosing zero anchors.

\emph{Spatial dimension $d\ge3$.} The $d\ne2$ argument in Theorem~\ref{thm:mass-mge3-full} uses the inequality $\dim W>2\dim C_S(t_0)$ for a regular semisimple $t_0$ acting invertibly on $W$. Here
\[
\dim W=2d>2+2\lfloor d/2\rfloor=2\dim C_S(t_0).
\]
Such a $t_0=ah+R_0$ exists: take $R_0$ regular semisimple and avoid the finitely many values of $a$ for which $a=0$ or $\rho(t_0)$ is singular. To recall the argument, write $T_0=\rho(t_0)$ and put
\[
E=\operatorname{span}\{T_0^{-1}\rho(b)v:b\in C_S(t_0)\}.
\]
The inequality implies that $E^{\perp_{\omega_M}}$ has dimension greater than half of $\dim W$. Since $B_{T_0}$ is a non-degenerate symmetric form, choose $u$ in this subspace with $B_{T_0}(u,u)\ne0$. A derivation representing $(t_0+u,y)$ satisfies, by Lemma~\ref{lem:mass-flexible-anchor}, $\lambda=0$ and $a_W=T_0^{-1}\rho(a_S)u$, where $a_S\in C_S(t_0)$. Thus
\[
q=\omega_M(a_W,v)
=\omega_M(u,T_0^{-1}\rho(a_S)v)=0.
\]

\emph{Spatial dimension $d=1$.} If $v=0$, choose $t_0=h$ and $u$ with $B_{\rho(h)}(u,u)\ne0$. The same anchor gives $\lambda=0$, and hence $q=0$. If $v\ne0$, choose $t_0\in\mathfrak{sl}_2$ with
\[
\det\rho_1(t_0)\ne0,\qquad
\omega_M(\rho_1(t_0)v,v)\ne0.
\]
This is possible because the second expression is a non-zero linear functional of $t_0$ (the space $\mathfrak{sl}_2v$ equals $U_1$), and the two non-empty Zariski open conditions intersect. A trace-zero invertible $2\times2$ matrix satisfies $T_0^{-1}=-T_0/\det T_0$, so $B_{T_0}(v,v)\ne0$. Use the zero anchor $t_0+v$. Lemma~\ref{lem:mass-flexible-anchor} gives $\lambda=0$ and $a_W=T_0^{-1}\rho(a_S)v$. Since $C_{\mathfrak{sl}_2}(t_0)=\C t_0$, we get $a_W=\beta v$ and $q=\beta\omega_M(v,v)=0$.

\emph{Spatial dimension $d=2$.} Use $p_j=P_1^{(j)}+\mathrm iP_2^{(j)}$ and $q_j=P_1^{(j)}-\mathrm iP_2^{(j)}$, for $j=0,1$. Then
\[
\omega_M(p_0,q_1)=-2,\qquad
\omega_M(p_1,q_0)=2,
\]
and the other independent pairings vanish. In the ordered basis $(p_0,p_1,q_0,q_1)$, the action of $t_0=h+\tau s$ is
\[
T_0=\operatorname{diag}(1+\tau,-1+\tau,1-\tau,-1-\tau).
\]
Conjugation by $SL_2(\C)$ fixes $s,M$, preserves $\omega_M$, and preserves the normalization of $\Delta$. Its diagonal action on the two vectors of $U_1$ puts $v$ in one of the following forms:
\[
\begin{array}{c|l}
\mathrm{(i)}&p_0+\kappa q_1\quad(\kappa\ne0),\ \text{or }p_0\ (\kappa=0),\\
\mathrm{(ii)}&p_0+p_1+\kappa(q_0+q_1)\quad(\kappa\ne0),\\
\mathrm{(iii)}&q_0,\\
\mathrm{(iv)}&0.
\end{array}
\]
Indeed, independent string vectors can be sent to $(u_0,\kappa u_1)$ by a determinant-one transformation; dependent non-zero string vectors can be sent to $(u_0+u_1,\kappa(u_0+u_1))$. The two remaining non-zero cases have just one string and can be sent to $p_0$ or $q_0$.

For an anchor $r=t_0+u$, a representing derivation with $D(r)=0$ has
\[
a_S=\alpha h+\gamma s,\qquad
a_W=T_0^{-1}(A+\lambda I)u,\qquad
\lambda B+\tau\mu=0,
\]
where $A=\rho(a_S)$ and $B=B_{T_0}(u,u)$. Write $\sigma$ for the $s$-coordinate of $t$. The central component at $y$ is
\[
q=\omega_M(a_W,v)+2c\lambda+\sigma\mu.
\]
First suppose $\sigma\ne0$, and choose $\tau\notin\{0,1,-1\}$. Then $\mu=-\lambda B/\tau$. In cases (i) and (iv), take
\[
u=Xp_1+q_0,\qquad X=\frac{c\tau(\tau-1)}{2\sigma}.
\]
Here $B=4X/(\tau-1)$ and $\omega_M(a_W,v)=0$, so $q=0$. In case (iii), take
\[
u=Xp_0+q_1,\qquad X=-\frac{c\tau(1+\tau)}{2\sigma}.
\]
Now $B=-4X/(1+\tau)$ and again $\omega_M(a_W,v)=0$, giving $q=0$. In case (ii), take $u=x(p_0+\kappa q_1)$. Direct substitution gives
\[
B=-\frac{4\kappa x^2}{1+\tau},\qquad
\omega_M(a_W,v)=-\frac{4\kappa x\lambda}{1+\tau}.
\]
Consequently
\[
q=\frac{\lambda}{\tau(1+\tau)}
\bigl(4\sigma\kappa x^2-4\kappa\tau x+2c\tau(1+\tau)\bigr).
\]
Choose a root of the displayed quadratic, which exists over $\C$ since $\sigma\kappa\ne0$. Its constant term is non-zero, so the root is non-zero. This gives $q=0$.

Finally suppose $\sigma=0$. Set $\tau=0$. Use $u=p_1+q_0$ in cases (i) and (iv), $u=p_0+q_1$ in case (iii), and $u=p_0+\kappa q_1$ in case (ii). The $B$-value is $-4$ in cases (i), (iii), and (iv), and $-4\kappa$ in case (ii); all are non-zero. Thus $\lambda=0$. In cases (i), (iii), and (iv), the pairing $\omega_M(a_W,v)$ vanishes by the supports of the vectors. In case (ii), it equals $-4\kappa\lambda=0$. The shear term vanishes because $\sigma=0$, so $q=0$ in every case. Therefore the normalized map vanishes identically, completing the endpoint proof.
\end{proof}

\begin{proposition}\label{prop:exotic-center-free-hyperplane}
Let $\Delta$ be a 2-local derivation on $\widehat{\mathfrak g}^{\,\Theta}_\ell(2)$, normalized as in Proposition~\ref{prop:central-center-valued-reduction}.  Subtract a scalar multiple of $\eta_\Theta$ so that $\Delta(s)=0$.  Then
\[
\Delta(x)=0
\qquad
\text{for every }x\text{ whose }z\text{-coordinate is zero}.
\]
\end{proposition}
\begin{proof}
Let $x$ have zero $z$-coordinate and represent the pair $(s,x)$ by
\[
D=\operatorname{ad}(a)+\lambda\delta_\Theta+\mu\eta_\Theta.
\]
Since $D(s)=\Delta(s)=0$,
\[
[a_V,s]+\mu z=0.
\]
In the $p$--$q$ basis the operator $\operatorname{ad}s$ has eigenvalue $1$ on every $p_k$ and eigenvalue $-1$ on every $q_k$.  Hence the non-central radical part of $a$ is zero, and comparison of the central component gives $\mu=0$.  The remaining derivation $\operatorname{ad}(a_S)+\lambda\delta_\Theta$ has no central component on any vector with zero $z$-coordinate.  Since $D(x)=\Delta(x)$ is central, it must be zero.
\end{proof}

\subsection{Middle-weight and affine zero anchors in the exotic extension}
We use a middle-weight stabilizer, a conjugate positioning argument, and one affine zero anchor to eliminate all central translates.

Put $m=2\ell$ and
\[
W_\Theta=\operatorname{span}\{p_0,\ldots,p_m,q_0,\ldots,q_m\}.
\]
The Heisenberg bracket defines a non-degenerate alternating form $\omega_\Theta$ by
\[
[u,v]=\omega_\Theta(u,v)z,
\qquad u,v\in W_\Theta.
\]

\begin{lemma}\label{lem:exotic-middle-anchor}
Let
\[
u=a p_\ell+bq_\ell,
\qquad ab\ne0,
\qquad r=h+u.
\]
If
\[
D=\operatorname{ad}(a_S+a_W)+\lambda\delta_\Theta+\mu\eta_\Theta
\]
satisfies $D(r)=0$, then $\lambda=0$ and, up to the ineffective central summand in the inner element,
\begin{equation}\label{eq:exotic-middle-anchor-stabilizer}
D=\operatorname{ad}(\alpha h+\beta u)+\mu\eta_\Theta
\end{equation}
for some $\alpha,\beta,\mu\in\mathbb C$.
\end{lemma}
\begin{proof}
Write
\[
a_S=\alpha_e e+\alpha_f f+\alpha_h h+\alpha_s s.
\]
The $\mathfrak{sl}_2$-component of $D(h+u)=0$ gives
\[
\alpha_e=\alpha_f=0.
\]
Since the $h$-weights on $W_\Theta$ are non-zero away from the middle level, all components of $a_W$ outside
\[
\mathbb Cp_\ell\oplus\mathbb Cq_\ell
\]
vanish.  Write $a_W=A p_\ell+Bq_\ell$.  The $p_\ell$- and $q_\ell$-components of $D(r)=0$ are
\[
a(\alpha_s+\lambda)=0,
\qquad
b(-\alpha_s+\lambda)=0.
\]
Because $ab\ne0$, we get
\[
\alpha_s=\lambda=0.
\]
The central component of $D(r)=0$ is
\[
\omega_\Theta(a_W,u)=0.
\]
The restriction of $\omega_\Theta$ to
$\mathbb Cp_\ell\oplus\mathbb Cq_\ell$ is non-degenerate, and the symplectic orthogonal complement of the line $\mathbb Cu$ inside this two-dimensional plane is the same line.  Hence $a_W=\beta u$.  The shear vanishes on $h+u$, so its coefficient remains free.  This proves \eqref{eq:exotic-middle-anchor-stabilizer}.
\end{proof}

\begin{lemma}\label{lem:exotic-middle-positioning}
Let
\[
v=v_p+v_q\in W_\Theta,
\qquad
v_p\in\operatorname{span}\{p_0,\ldots,p_m\},
\quad
v_q\in\operatorname{span}\{q_0,\ldots,q_m\}.
\]
There exists an inner automorphism generated by the $\mathfrak{sl}_2$-factor, and hence a conjugate triple $(e',h',f')$ with corresponding weight bases $p_0',\ldots,p_m'$ and $q_0',\ldots,q_m'$, such that the coefficients $A,B$ of $v$ along $p_\ell',q_\ell'$ satisfy either
\[
A=B=0
\]
or
\[
AB\ne0.
\]
\end{lemma}
\begin{proof}
Each of the two strings is an irreducible $U_m$-module.  Fix the middle-coordinate functional on $U_m$.  For a non-zero vector $w\in U_m$, the corresponding matrix coefficient on $SL_2(\mathbb C)$,
\[
g\longmapsto \operatorname{coeff}_{u_\ell}(g^{-1}w),
\]
is a non-zero regular function: its non-vanishing follows because the orbit of $w$ spans the irreducible module $U_m$.  If both $v_p$ and $v_q$ are non-zero, the complements of their two zero sets are non-empty Zariski open subsets of the irreducible group $SL_2(\mathbb C)$, so their intersection is non-empty; choose $g$ there and obtain $AB\ne0$.

Suppose exactly one string is non-zero, and denote its non-zero vector by $w$.  We claim that its middle-coordinate matrix coefficient has a zero.  Assume the contrary.  For every $g\in SL_2(\C)$ the polynomial
\[
P_g(t)=\varphi\bigl(\exp(-te)g^{-1}w\bigr),
\qquad
\varphi=\operatorname{coeff}_{u_\ell},
\]
has no zero in $\C$. Since a non-zero complex polynomial with no zero is constant, $P_g$ is constant.  Differentiating at $t=0$ gives
\[
\varphi(e g^{-1}w)=0
\qquad(g\in SL_2(\C)).
\]
The linear span of the orbit $\{g^{-1}w:g\in SL_2(\C)\}$ is a non-zero invariant subspace of the irreducible module $U_m$, and hence is all of $U_m$.  Thus $\varphi\circ e=0$ on $U_m$.  This is impossible because
\[
e u_{\ell+1}=(\ell+1)u_\ell
\]
and $\ell\ge1$.  Therefore some $g$ makes the middle coefficient zero.  The absent string has middle coefficient zero automatically, so $A=B=0$.  The case $v=0$ is immediate.  Conjugating the standard triple and the two weight bases by this $g$ gives the assertion.
\end{proof}

\begin{lemma}\label{lem:exotic-affine-middle-anchor}
Fix a conjugate triple and corresponding middle vectors $p_\ell',q_\ell'$.  Let
\[
r=h'+\tau s+u,
\qquad
u=ap_\ell'+bq_\ell',
\qquad
\tau ab\ne0,
\]
and assume that the action of $h'+\tau s$ on $W_\Theta$ is invertible.  If
\[
D=\operatorname{ad}(a_S+a_W)+\lambda\delta_\Theta+\mu\eta_\Theta
\]
satisfies $D(r)=0$, then, writing $c_\ell=(-1)^\ell(\ell!)^2$, its coefficients satisfy
\begin{equation}\label{eq:exotic-affine-aW}
a_S=\alpha h'+\gamma s,
\qquad
a_W=\frac{(\gamma+\lambda)a}{\tau}p_\ell'
     +\frac{(\gamma-\lambda)b}{\tau}q_\ell'
\end{equation}
up to an ineffective central summand, and
\begin{equation}\label{eq:exotic-affine-mu}
\mu=-\frac{2c_\ell ab}{\tau^2}\lambda.
\end{equation}
\end{lemma}
\begin{proof}
The reductive component of $D(r)=0$ gives $a_S\in C_{\mathfrak{sl}_2\oplus\mathbb Cs}(h'+\tau s)=\mathbb Ch'\oplus\mathbb Cs$.  Put $a_S=\alpha h'+\gamma s$.  On the middle weight lines, $h'$ acts by zero, while $s$ has eigenvalues $1$ on $p_\ell'$ and $-1$ on $q_\ell'$.  The $W_\Theta$-component of $D(r)=0$ is therefore
\[
(\rho(a_S)+\lambda I)u-\rho(h'+\tau s)a_W=0.
\]
Invertibility of $\rho(h'+\tau s)$ gives \eqref{eq:exotic-affine-aW}.  Since
\[
[p_\ell',q_\ell']=c_\ell z,
\]
the central component becomes
\[
0=\omega_\Theta(a_W,u)+\tau\mu
 =\frac{2c_\ell ab}{\tau}\lambda+\tau\mu,
\]
which is \eqref{eq:exotic-affine-mu}.
\end{proof}

\begin{theorem}\label{thm:exotic-all}
For every $\ell\in\mathbb N$,
\begin{equation}\label{eq:exotic-all}
\TwoLocDer\widehat{\mathfrak g}^{\,\Theta}_\ell(2)
=\Der\widehat{\mathfrak g}^{\,\Theta}_\ell(2).
\end{equation}
\end{theorem}
\begin{proof}
Normalize $\Delta$ as in Proposition~\ref{prop:exotic-center-free-hyperplane}.  Thus $\Delta$ is center-valued and vanishes on every element with zero $z$-coordinate.  Consider an arbitrary
\[
y=x_0+\sigma s+v+cz,
\qquad
x_0\in\mathfrak{sl}_2,
\quad v\in W_\Theta.
\]
Write $\Delta(y)=qz$.  The case $c=0$ is already settled.

Apply Lemma~\ref{lem:exotic-middle-positioning} to $v$ and use the resulting conjugate triple.  Let $A,B$ denote the middle coefficients of $v$ along $p_\ell',q_\ell'$.  Thus either $A=B=0$ or $AB\ne0$.

First suppose $\sigma=0$.  Choose
\[
u=ap_\ell'+bq_\ell',
\qquad ab\ne0,
\qquad \omega_\Theta(u,v)=0.
\]
If $A=B=0$, any $a,b\ne0$ work; if $AB\ne0$, take $u=Ap_\ell'+Bq_\ell'$.  The conjugate form of Lemma~\ref{lem:exotic-middle-anchor} applies to the zero anchor $h'+u$.  Any derivation representing the pair $(h'+u,y)$ has zero scalar outer coefficient and has radical inner part proportional to $u$.  Since the $s$-coordinate of $y$ is zero, the shear contributes nothing, while the only possible central inner term is proportional to $\omega_\Theta(u,v)=0$.  Hence $q=0$.

Now suppose $\sigma\ne0$.  We may assume $c\ne0$.  Choose $\tau\ne0$ outside the finite set for which $\rho(h'+\tau s)$ is singular on $W_\Theta$.  Let
\[
r=h'+\tau s+ap_\ell'+bq_\ell'.
\]
This anchor has zero $z$-coordinate, so $\Delta(r)=0$.  Let $D$ represent $(r,y)$.  By Lemma~\ref{lem:exotic-affine-middle-anchor}, the central coefficient of $D(y)$ is
\begin{equation}\label{eq:exotic-final-central-q}
q=\omega_\Theta(a_W,v)+2c\lambda+\sigma\mu.
\end{equation}

If $A=B=0$, the first term in \eqref{eq:exotic-final-central-q} is zero.  Choose non-zero $a,b$ with
\[
ab=\frac{c\tau^2}{\sigma c_\ell}.
\]
Then \eqref{eq:exotic-affine-mu} gives
\[
q=2c\lambda-\frac{2\sigma c_\ell ab}{\tau^2}\lambda=0.
\]

If $AB\ne0$, put
\[
a=\kappa A,
\qquad
b=\kappa B,
\qquad
\kappa\ne0.
\]
The $\gamma$-terms cancel in the symplectic pairing, and \eqref{eq:exotic-affine-aW} gives
\[
\omega_\Theta(a_W,v)
 =\frac{2c_\ell\kappa AB}{\tau}\lambda.
\]
Together with \eqref{eq:exotic-affine-mu}, formula \eqref{eq:exotic-final-central-q} becomes
\[
q=2\lambda\left[
 c+c_\ell AB\left(
 \frac{\kappa}{\tau}-\frac{\sigma\kappa^2}{\tau^2}
 \right)
\right].
\]
Choose $\kappa$ to be a root of
\begin{equation}\label{eq:exotic-final-kappa}
c_\ell AB\left(
 \frac{\kappa}{\tau}-\frac{\sigma\kappa^2}{\tau^2}
\right)+c=0.
\end{equation}
The polynomial in \eqref{eq:exotic-final-kappa} is quadratic with non-zero leading and constant terms, so it has a non-zero root over $\mathbb C$.  For this choice $q=0$.

Thus every value of the normalized 2-local derivation is zero.  Restoring the global derivations removed during normalization proves \eqref{eq:exotic-all}.
\end{proof}

\begin{remark}\label{rem:central-translate-reduction}
The final mass and exotic arguments illustrate the genuinely two-local step that is invisible from center normalization alone.  A residual map may vanish both on the center and on the complementary zero-central hyperplane without being additive.  The flexible symplectic anchor in the mass case and the affine middle anchor in the exotic case supply a second zero point whose stabilizer forces the value on every central translate to vanish.
\end{remark}

\section{Main classification and consequences}\label{sec:mainclassification}
We collect the three rigidity theorems in a single statement.

\begin{theorem}\label{thm:main-classification}
For all admissible parameters, every 2-local derivation on the conformal Galilei algebra, on its mass central extension, and on its exotic central extension is a derivation.  Explicitly,
\begin{align}
\TwoLocDer\mathfrak g_\ell(d)
 &=\Der\mathfrak g_\ell(d),
 &&d\ge1,\quad \ell\in\tfrac12\mathbb N,\label{eq:main-noncentral}\\
\TwoLocDer\widehat{\mathfrak g}^{\,M}_\ell(d)
 &=\Der\widehat{\mathfrak g}^{\,M}_\ell(d),
 &&d\ge1,\quad \ell\in\mathbb N-\tfrac12,\label{eq:main-mass}\\
\TwoLocDer\widehat{\mathfrak g}^{\,\Theta}_\ell(2)
 &=\Der\widehat{\mathfrak g}^{\,\Theta}_\ell(2),
 &&\ell\in\mathbb N.\label{eq:main-exotic}
\end{align}
\end{theorem}
\begin{proof}
Equation~\eqref{eq:main-noncentral} is Corollary~\ref{cor:noncentral-all}.  Equation~\eqref{eq:main-mass} is Theorem~\ref{thm:mass-all}.  Equation~\eqref{eq:main-exotic} is Theorem~\ref{thm:exotic-all}.
\end{proof}

\begin{corollary}\label{cor:automatic-linearity}
Every 2-local derivation in the three families above is automatically additive, homogeneous, and linear.
\end{corollary}
\begin{proof}
By Theorem~\ref{thm:main-classification}, every 2-local derivation is an ordinary derivation, and derivations are linear.
\end{proof}

Let
\[
r_d=\frac{d(d-1)}2,
\qquad m=2\ell.
\]
Since the sets of 2-local derivations are now identified with derivation algebras, their dimensions and Lie-algebra structures follow from the derivation decompositions in Section~\ref{sec:der2local}.

\begin{corollary}\label{cor:dimensions}
The dimensions of the 2-local derivation sets, viewed through their identification with the corresponding derivation algebras, are
\begin{align*}
\dim\TwoLocDer\mathfrak g_\ell(d)
 &=4+r_d+d(m+1),\\
\dim\TwoLocDer\widehat{\mathfrak g}^{\,M}_\ell(d)
 &=4+r_d+d(m+1), && d\ne2,\\
\dim\TwoLocDer\widehat{\mathfrak g}^{\,M}_\ell(2)
 &=2m+8=4\ell+8,\\
\dim\TwoLocDer\widehat{\mathfrak g}^{\,\Theta}_\ell(2)
 &=4\ell+8.
\end{align*}
In particular, each 2-local derivation set is a Lie algebra under the commutator because it is exactly the corresponding derivation algebra.
\end{corollary}
\begin{proof}
For the non-central algebra, $\mathfrak g_\ell(d)$ is centerless and
\[
\Der\mathfrak g_\ell(d)=\Inn\mathfrak g_\ell(d)\oplus\mathbb C\delta_0,
\]
so its derivation dimension is $\dim\mathfrak g_\ell(d)+1=4+r_d+d(m+1)$.  In the mass case the center has dimension one.  Thus $\dim\Inn L=\dim L-1$; adding $\delta_M$ gives $4+r_d+d(m+1)$ when $d\ne2$, while the additional shear $\eta_M$ contributes one more dimension for $d=2$.  In the exotic case $\dim L=4\ell+7$ and $\dim Z(L)=1$, so $\dim\Inn L=4\ell+6$; adding the two independent outer directions gives $\dim\Der L=4\ell+8$.
\end{proof}

\section{Conclusion}
Together with the local classification in \cite{ASYLocalCG}, our theorem distinguishes one-point from two-point interpolation: proper local derivations occur in the exceptional unextended family, whereas 2-local derivations are rigid throughout.

We have obtained a complete 2-local derivation classification for the finite-dimensional conformal Galilei Lie algebras in the non-central, mass-central, and exotic-central settings.  In every admissible case the two-point condition forces a single global derivation, even though no linearity is assumed initially.  The proof also identifies three mechanisms that may be useful for other non-semisimple Lie algebras: semisimple perturbation of a residual rotation, symplectic zero anchors in a Heisenberg radical, and affine weight anchors adapted to a central cocycle.  These methods separate the quotient-rigidity step from the genuinely nonlinear problem of controlling central translates.

\section{Declarations}
\paragraph{Funding.}
The authors declare that no funds, grants, or other financial support were received for the preparation of this manuscript.

\paragraph{Conflict of interest.}
The authors declare that they have no conflict of interest.

\paragraph{Data availability.}
No datasets were generated or analyzed during the current study.

\paragraph{Author contributions.}
All authors contributed to the conception, mathematical analysis, and preparation of the manuscript. All authors reviewed and approved the manuscript.

\end{document}